\documentclass[final,nomarks]{dmtcs-episciences}%
\usepackage[utf8]{inputenc}
\usepackage{subfigure}
\usepackage{amssymb}
\usepackage{amsmath}
\usepackage{amsthm}
\usepackage{tikz}
\usepackage{needspace}
\usepackage[round]{natbib}
\usepackage{graphicx}
\usepackage{amsfonts}%
\providecommand{\U}[1]{\protect\rule{.1in}{.1in}}
\usetikzlibrary{arrows}
\newtheorem{theorem}{Theorem}[section]
\newtheorem{lemma}[theorem]{Lemma}

\newtheorem{proposition}[theorem]{Proposition}
\newtheorem{question}[theorem]{Question}
\newtheorem{remark}[theorem]{Remark}
\theoremstyle{definition}
\newtheorem{definition}[theorem]{Definition}
\newtheorem{example}[theorem]{Example}
\newenvironment{statement}{\begin{quote}}{\end{quote}}

\author{Darij Grinberg\affiliationmark{1}}
\title{The Elser nuclei sum revisited}

\affiliation{Drexel University, Philadelphia, PA, USA}

\received{2020-12-22}
\revised{2021-01-12, 2021-04-28}
\accepted{2021-05-04}

\begin{document}

\publicationdetails{23}{2021}{1}{15}{7012}
\maketitle

\begin{abstract}
Fix a finite undirected graph $\Gamma$ and a vertex $v$ of $\Gamma$. Let $E$
be the set of edges of $\Gamma$. We call a subset $F$ of $E$ \emph{pandemic}
if each edge of $\Gamma$ has at least one endpoint that can be connected to
$v$ by an $F$-path (i.e., a path using edges from $F$ only). In 1984, Elser
showed that the sum of $\left(  -1\right)  ^{\left\vert F\right\vert }$ over
all pandemic subsets $F$ of $E$ is $0$ if $E\neq\varnothing$. We give a simple
proof of this result via a sign-reversing involution, and discuss variants,
generalizations and a refinement using discrete Morse theory.

\end{abstract}

In \cite{Elser84}, Veit Elser studied the probabilities of clusters forming
when $n$ points are sampled randomly in a $d$-dimensional volume. In the
process, he found a purely graph-theoretical lemma \cite[Lemma 1]{Elser84},
which served a crucial role in his work. For decades, the lemma stayed hidden
from the eyes of combinatorialists in a physics journal, until it resurfaced
in recent work \cite{DHLetc19} by Dorpalen-Barry, Hettle, Livingston, Martin,
Nasr, Vega and Whitlatch. In this note, I will show a simpler proof of the
lemma using a sign-reversing involution. The proof also suggests multiple
venues of generalization that I will explore in the later sections; one
extends the lemma to a statement about arbitrary antimatroids (and even a
wider setting). Finally, I will strengthen the lemma to a Morse-theoretical
result, stating the collapsibility of a certain simplicial complex. Some open
questions will be posed.

\subsection*{Remark on alternative versions}

The current version of this paper is written with a combinatorially
experienced reader in mind. The previous arXiv version \cite{vershort}
includes more details in the proofs.

\section{\label{sec.elser}Elser's result}

Let us first introduce our setting, which is slightly more general (and
perhaps also simpler) than that used in \cite{Elser84}. (In Section
\ref{sec.abstract}, we will move to a more general setup.)

We fix an arbitrary graph $\Gamma$ with vertex set $V$ and edge set $E$. Here,
``graph'' means ``finite
undirected multigraph'' -- i.e., it can have self-loops and
parallel edges, but it has finitely many vertices and edges, and its edges are undirected.

We fix a vertex $v\in V$.

If $F\subseteq E$, then an $F$\emph{-path} shall mean a path of $\Gamma$ such
that all edges of the path belong to $F$.

If $e\in E$ is any edge and $F\subseteq E$ is any subset, then we say that $F$
\emph{infects }$e$ if there exists an $F$-path from $v$ to some endpoint of
$e$. (The terminology is inspired by the idea of an infectious disease
starting in the vertex $v$ and being transmitted along edges.)\footnote{Note
that if an edge $e$ contains the vertex $v$, then any subset $F$ of $E$ (even
the empty one) infects $e$, since there is a trivial (edgeless) $F$-path from
$v$ to $v$.}

A subset $F\subseteq E$ is said to be \emph{pandemic} if it infects each edge
$e\in E$.

\begin{example}
\label{exa.example1} Let $\Gamma$ be the following graph:%
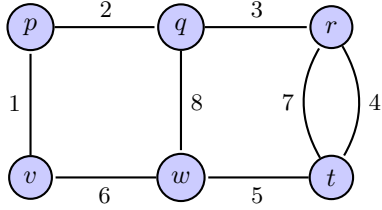
\[%
\begin{tikzpicture}%
[-,>=stealth',shorten >=1pt,auto,node distance=2cm, thick,main node/.style={circle,fill=blue!20,draw}%
]
\node[main node] (1) {$v$};
\node[main node] [above of=1] (2) {$p$};
\node[main node] [right of=1] (3) {$w$};
\node[main node] [above of=3] (4) {$q$};
\node[main node] [right of=3] (5) {$t$};
\node[main node] [above of=5] (6) {$r$};
\path[every node/.style={font=\sffamily\small}] (1) edge node {$1$}
(2) (2) edge node {$2$} (4) (4) edge node {$3$}
(6) (6) edge [bend left] node {$4$} (5) (5) edge node {$5$}
(3) (3) edge node {$6$} (1) (5) edge [bend left] node {$7$}
(6) (4) edge node {$8$} (3);
\end{tikzpicture}%
\]
(where the vertex $v$ is the vertex labelled $v$). Then, for example, the set
$\left\{  1,2\right\}  \subseteq E$ infects edges $1,2,3,6,8$ (but none of the
other edges). The set $\left\{  1,2,5\right\}  $ infects the same edges as
$\left\{  1,2\right\}  $ (indeed, the additional edge $5$ does not increase
its infectiousness, since it is not on any $\left\{  1,2,5\right\}  $-path
from $v$). The set $\left\{  1,2,3\right\}  $ infects every edge other than
$5$. The set $\left\{  1,2,3,4\right\}  $ infects each edge, and thus is pandemic.
\end{example}

Now, we can state our version of \cite[Lemma 1]{Elser84}:

\begin{theorem}
\label{thm.elser0}Assume that $E\neq\varnothing$. Then,%
\begin{equation}
\sum_{\substack{F\subseteq E\text{ is}\\\text{pandemic}}}\left(  -1\right)
^{\left\vert F\right\vert }=0. \label{eq.thm.elser0.eq}%
\end{equation}

\end{theorem}

\begin{example}
\label{exa.example2} Let $\Gamma$ be the following graph:%
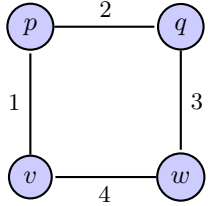
\[%
\begin{tikzpicture}%
[-,>=stealth',shorten >=1pt,auto,node distance=2cm, thick,main node/.style={circle,fill=blue!20,draw}%
]
\node[main node] (1) {$v$};
\node[main node] [above of=1] (2) {$p$};
\node[main node] [right of=2] (3) {$q$};
\node[main node] [right of=1] (4) {$w$};
\path[every node/.style={font=\sffamily\small}] (1) edge node {$1$}
(2) (2) edge node {$2$} (3) (3) edge node {$3$} (4) (4) edge node {$4$} (1);
\end{tikzpicture}%
\]
(where the vertex $v$ is the vertex labelled $v$). Then, the pandemic subsets
of $E$ are the sets%
\[
\left\{  1,2\right\}  ,\ \ \left\{  1,4\right\}  ,\ \ \left\{  3,4\right\}
,\ \ \left\{  1,2,3\right\}  ,\ \ \left\{  1,3,4\right\}  ,\ \ \left\{
1,2,4\right\}  ,\ \ \left\{  2,3,4\right\}  ,\ \ \left\{  1,2,3,4\right\}  .
\]
The sizes of these subsets are $2,2,2,3,3,3,3,4$, respectively. Hence,
(\ref{eq.thm.elser0.eq}) says that%
\[
\left(  -1\right)  ^{2}+\left(  -1\right)  ^{2}+\left(  -1\right)
^{2}+\left(  -1\right)  ^{3}+\left(  -1\right)  ^{3}+\left(  -1\right)
^{3}+\left(  -1\right)  ^{3}+\left(  -1\right)  ^{4}=0.
\]

\end{example}

We note that the equality (\ref{eq.thm.elser0.eq}) can be restated as
``there are equally many pandemic subsets $F\subseteq E$ of
even size and pandemic subsets $F\subseteq E$ of odd size''.
Thus, in particular, the number of all pandemic subsets $F$ of $E$ is even
(when $E\neq\varnothing$).

\begin{remark}
Theorem \ref{thm.elser0} is a bit more general than \cite[Lemma 1]{Elser84}.
To see why, we assume that the graph $\Gamma$ is connected and simple (i.e.,
has no self-loops and parallel edges). Then, a \emph{nucleus} is defined in
\cite{Elser84} as a subgraph $N$ of $\Gamma$ with the properties that

\begin{enumerate}
\item the subgraph $N$ is connected, and

\item each edge of $\Gamma$ has at least one endpoint in $N$.
\end{enumerate}

\noindent Given a subgraph $N$ of $\Gamma$, we let $\operatorname*{E}\left(
N\right)  $ denote the set of all edges of $N$. Now, \cite[Lemma 1]{Elser84}
claims that if $E\neq\varnothing$, then%
\[
\sum_{\substack{N\text{ is a nucleus}\\\text{containing }v}}\left(  -1\right)
^{\left\vert \operatorname*{E}\left(  N\right)  \right\vert }=0.
\]

But this is equivalent to (\ref{eq.thm.elser0.eq}), because there is a
bijection%
\begin{align*}
\left\{  \text{nuclei containing }v\right\}   &  \rightarrow\left\{
\text{pandemic subsets }F\subseteq E\right\}  ,\\
N  &  \mapsto\operatorname*{E}\left(  N\right)  .
\end{align*}
We leave it to the reader to check this in detail; what needs to be checked
are the following three statements:

\begin{itemize}
\item If $N$ is a nucleus containing $v$, then $\operatorname*{E}\left(
N\right)  $ is a pandemic subset of $E$.

\item Every nucleus $N$ containing $v$ is uniquely determined by the set
$\operatorname*{E}\left(  N\right)  $. (Indeed, since a nucleus has to be
connected, each of its vertices must be an endpoint of one of its edges,
unless its only vertex is $v$.)

\item If $F$ is a pandemic subset of $E$, then there is a nucleus $N$
containing $v$ such that $\operatorname*{E}\left(  N\right)  =F$. (Indeed, $N$
can be defined as the subgraph of $\Gamma$ whose vertices are the endpoints of
all edges in $F$ as well as the vertex $v$, and whose edges are the edges in
$F$. To see that this subgraph $N$ is connected, it suffices to argue that
each of its vertices has a path to $v$; but this follows from the definition
of ``pandemic'', since each vertex of $N$
other than $v$ belongs to at least one edge in $F$.)
\end{itemize}

Thus, Theorem \ref{thm.elser0} is equivalent to \cite[Lemma 1]{Elser84} in the
case when $\Gamma$ is connected and simple.
\end{remark}

\begin{remark}
It might appear more natural to talk about a subset $F\subseteq E$ infecting a
vertex rather than an edge. (Namely, we can say that $F$ infects a vertex $w$
if there is an $F$-path from $v$ to $w$.) However, the analogue of Theorem
\ref{thm.elser0} in which pandemicity is defined via infecting all vertices is
not true. The graph of Example \ref{exa.example2} provides a counterexample.
\end{remark}

\section{The proof}

\subsection{\label{subsect.sets}Set-theoretical notions}

We shall first introduce some concepts and notations pertaining to arbitrary
sets. They will aid us in proving Theorem \ref{thm.elser0}, and also in
generalizing it later on.

\begin{definition}
\label{def.sets.prec-sets}Let $A$ and $B$ be two sets. Then, we say that
$A\prec B$ if we have $B=A\cup\left\{  b\right\}  $ for some $b\in B\setminus
A$. Equivalently, $A\prec B$ holds if and only if $A\subseteq B$ and
$\left\vert B\setminus A\right\vert =1$.
\end{definition}

\begin{definition}
\label{def.sets.PE}We let $\mathcal{P}\left(  E\right)  $ denote the power set
of $E$ (that is, the set of all subsets of $E$).
\end{definition}

\begin{definition}
\label{def.sets.invol}Let $\mathcal{A}$ be a set. A map $\mu:\mathcal{A}%
\rightarrow\mathcal{A}$ is said to be an \emph{involution} if $\mu\circ
\mu=\operatorname*{id}\nolimits_{\mathcal{A}}$.
\end{definition}

\begin{definition}
\label{def.sets.comp-match}Let $\mathcal{A}$ be a subset of $\mathcal{P}%
\left(  E\right)  $. A \emph{complete matching} of $\mathcal{A}$ shall mean an
involution $\mu:\mathcal{A}\rightarrow\mathcal{A}$ with the property that each
$F\in\mathcal{A}$ satisfies%
\begin{equation}
\text{either }\mu\left(  F\right)  \prec F\text{ or }F\prec\mu\left(
F\right)  . \label{eq.def.sets.comp-match.F-muF}%
\end{equation}

\end{definition}

The following simple fact abstracts an idea that will be used at least twice:

\begin{lemma}
\label{lem.sets.srvinv}Let $\mathcal{A}$ be a subset of $\mathcal{P}\left(
E\right)  $. Let $\mu:\mathcal{A}\rightarrow\mathcal{A}$ be a complete
matching of $\mathcal{A}$. Then,%
\begin{equation}
\sum_{F\in\mathcal{A}}\left(  -1\right)  ^{\left\vert F\right\vert }=0.
\label{eq.lem.sets.srvinv.eq}%
\end{equation}

\end{lemma}

\begin{proof}
This is a standard argument in enumerative combinatorics (see \cite[(2.3)]%
{Sagan20} or \cite{BenQui08} for a more general viewpoint). Here is the proof:
For each $F\in\mathcal{A}$, we have $\left\vert \mu\left(  F\right)
\right\vert =\left\vert F\right\vert \pm1$ (because
(\ref{eq.def.sets.comp-match.F-muF}) shows that the sets $F$ and $\mu\left(
F\right)  $ differ in exactly one element) and thus $\left(  -1\right)
^{\left\vert \mu\left(  F\right)  \right\vert }+\left(  -1\right)
^{\left\vert F\right\vert }=0$. This shows, in particular, that $\mu$ has no
fixed points. Thus, the involution $\mu$ partitions the set $\mathcal{A}$ into
$2$-element subsets $\left\{  F,\mu\left(  F\right)  \right\}  $. Each such
$2$-element subset contributes $\left(  -1\right)  ^{\left\vert \mu\left(
F\right)  \right\vert }+\left(  -1\right)  ^{\left\vert F\right\vert }=0$ to
the sum on the left hand side of (\ref{eq.lem.sets.srvinv.eq}). Hence, this
sum is $0$. This proves Lemma \ref{lem.sets.srvinv}.
\end{proof}

\subsection{Shades}

Next, we shall introduce the notion of a \emph{shade}; this will be crucial to
proving and generalizing Theorem \ref{thm.elser0}.

\begin{definition}
\label{def.Shade}Let $F$ be a subset of $E$. Then, we define a subset
$\operatorname*{Shade}F$ of $E$ by%
\begin{equation}
\operatorname*{Shade}F=\left\{  e\in E\ \mid\ F\text{ infects }e\right\}  .
\label{eq.def.Shade.eq}%
\end{equation}
We refer to $\operatorname*{Shade}F$ as the \emph{shade} of $F$.
\end{definition}

Thus, the shade of a subset $F\subseteq E$ is the set of all edges of $\Gamma$
that are infected by $F$. (In more standard graph-theoretical lingo, this
means that $\operatorname*{Shade}F$ is the set of edges that contain at least
one vertex of the connected component containing $v$ of the graph $\left(
V,F\right)  $.)

\begin{example}
In Example \ref{exa.example1}, we have $\operatorname*{Shade}\left\{
1,2\right\}  =\left\{  1,2,3,6,8\right\}  $ and $\operatorname*{Shade}\left\{
1\right\}  =\left\{  1,2,6\right\}  $ and $\operatorname*{Shade}\left\{
8\right\}  =\left\{  1,6\right\}  $.
\end{example}

The following property of shades is rather obvious:

\begin{lemma}
\label{lem.Shade-monoton}Let $A$ and $B$ be two subsets of $E$ such that
$A\subseteq B$. Then, $\operatorname*{Shade}A\subseteq\operatorname*{Shade}B$.
\end{lemma}

The major property of shades that we will need is the following:

\begin{lemma}
\label{lem.Shade-tog}Let $F$ be a subset of $E$. Let $u\in E$ be such that
$u\notin\operatorname*{Shade}F$. Then,%
\begin{equation}
\operatorname*{Shade}\left(  F\cup\left\{  u\right\}  \right)
=\operatorname*{Shade}F \label{eq.lem.Shade-tog.union}%
\end{equation}
and%
\begin{equation}
\operatorname*{Shade}\left(  F\setminus\left\{  u\right\}  \right)
=\operatorname*{Shade}F. \label{eq.lem.Shade-tog.diff}%
\end{equation}

\end{lemma}

\begin{proof}
There is no $F$-path from $v$ to any endpoint of $u$ (since $u\notin%
\operatorname*{Shade}F$). Hence, any $\left(  F\cup\left\{  u\right\}
\right)  $-path that starts at $v$ must be an $F$-path (as it would otherwise
use the edge $u$ and thus contain an $F$-path from $v$ to some endpoint of
$u$). This entails $\operatorname*{Shade}\left(  F\cup\left\{  u\right\}
\right)  \subseteq\operatorname*{Shade}F$. Combined with the opposite
inclusion (which follows from Lemma \ref{lem.Shade-monoton}), this yields
(\ref{eq.lem.Shade-tog.union}).

Also, Lemma \ref{lem.Shade-monoton} yields $\operatorname*{Shade}\left(
F\setminus\left\{  u\right\}  \right)  \subseteq\operatorname*{Shade}F$, so
that $u\notin\operatorname*{Shade}\left(  F\setminus\left\{  u\right\}
\right)  $. Hence, (\ref{eq.lem.Shade-tog.diff}) follows by applying
(\ref{eq.lem.Shade-tog.union}) to $F\setminus\left\{  u\right\}  $ instead of
$F$.
\end{proof}

\subsection{A slightly more general claim}

Lemma \ref{lem.Shade-tog} might not look very powerful, but it contains all we
need to prove Theorem \ref{thm.elser0}. Better yet, we shall prove the
following slightly more general version of Theorem \ref{thm.elser0}:

\begin{theorem}
\label{thm.elser-Shade1}Let $G$ be any subset of $E$. Assume that
$E\neq\varnothing$. Then,%
\[
\sum_{\substack{F\subseteq E;\\G\subseteq\operatorname*{Shade}F}}\left(
-1\right)  ^{\left\vert F\right\vert }=0.
\]

\end{theorem}

We will soon prove Theorem \ref{thm.elser-Shade1} and explain how Theorem
\ref{thm.elser0} follows from it. First, however, let us give an equivalent
(but slightly easier to prove) version of Theorem \ref{thm.elser-Shade1}:

\begin{theorem}
\label{thm.elser-Shade0}Let $G$ be any subset of $E$. Then,%
\begin{equation}
\sum_{\substack{F\subseteq E;\\G\not \subseteq \operatorname*{Shade}F}}\left(
-1\right)  ^{\left\vert F\right\vert }=0. \label{eq.thm.elser-Shade0.eq}%
\end{equation}

\end{theorem}

\begin{proof}
Let
\begin{equation}
\mathcal{A}=\left\{  F\subseteq E\ \mid\ G\not \subseteq \operatorname*{Shade}%
F\right\}  . \label{pf.thm.elser-Shade0.short.A=}%
\end{equation}
Thus, $\mathcal{A}$ is a subset of $\mathcal{P}\left(  E\right)  $, and each
$F\in\mathcal{A}$ satisfies $G\not \subseteq \operatorname*{Shade}F$.

We equip the finite set $E$ with a total order (chosen arbitrarily, but fixed
henceforth). If $F\in\mathcal{A}$, then we define $\varepsilon\left(
F\right)  $ to be the \textbf{smallest} edge $e\in G\setminus
\operatorname*{Shade}F$. (Such an edge exists, since $F\in\mathcal{A}$ entails
$G\not \subseteq \operatorname*{Shade}F$ and thus $G\setminus
\operatorname*{Shade}F\neq\varnothing$.)

For any $F\in\mathcal{A}$, we have $\varepsilon\left(  F\right)
\notin\operatorname*{Shade}F$ (by the definition of $\varepsilon\left(
F\right)  $). Thus, any $F\in\mathcal{A}$ satisfies $\operatorname*{Shade}%
\left(  F\cup\left\{  \varepsilon\left(  F\right)  \right\}  \right)
=\operatorname*{Shade}F$ (by (\ref{eq.lem.Shade-tog.union})) and
$\operatorname*{Shade}\left(  F\setminus\left\{  \varepsilon\left(  F\right)
\right\}  \right)  =\operatorname*{Shade}F$ (by (\ref{eq.lem.Shade-tog.diff}%
)). In other words, if we replace a set $F\in\mathcal{A}$ by $F\cup\left\{
\varepsilon\left(  F\right)  \right\}  $ or $F\setminus\left\{  \varepsilon
\left(  F\right)  \right\}  $, then $\operatorname*{Shade}F$ does not change.
Hence, $\varepsilon\left(  F\right)  $ does not change either (since
$\varepsilon\left(  F\right)  $ depends only on $\operatorname*{Shade}F$, but
not on $F$ itself). Furthermore, the resulting set ($F\cup\left\{
\varepsilon\left(  F\right)  \right\}  $ or $F\setminus\left\{  \varepsilon
\left(  F\right)  \right\}  $) still belongs to $\mathcal{A}$ (since
$\operatorname*{Shade}F$ has not changed). Thus, we can define a map%
\begin{align*}
\mu:\mathcal{A}  &  \rightarrow\mathcal{A},\\
F  &  \mapsto%
\begin{cases}
F\cup\left\{  \varepsilon\left(  F\right)  \right\}  , & \text{if }%
\varepsilon\left(  F\right)  \notin F;\\
F\setminus\left\{  \varepsilon\left(  F\right)  \right\}  , & \text{if
}\varepsilon\left(  F\right)  \in F.
\end{cases}
\end{align*}
Clearly, this map $\mu$ is an involution (since we have shown that
$\varepsilon\left(  F\right)  $ does not change when we replace $F$ by
$\mu\left(  F\right)  $). Moreover, this map $\mu$ is a complete matching
(since each $F\in\mathcal{A}$ satisfies $F\prec\mu\left(  F\right)  $ if
$\varepsilon\left(  F\right)  \notin F$, and satisfies $\mu\left(  F\right)
\prec F$ otherwise). Hence, Lemma \ref{lem.sets.srvinv} yields $\sum
\limits_{F\in\mathcal{A}}\left(  -1\right)  ^{\left\vert F\right\vert }=0$. In
view of how we defined $\mathcal{A}$, this is equivalent to
(\ref{eq.thm.elser-Shade0.eq}). Thus, (\ref{eq.thm.elser-Shade0.eq}) is proven.
\end{proof}

In order to derive Theorem \ref{thm.elser-Shade1} from Theorem
\ref{thm.elser-Shade0}, we need the following innocent lemma:

\begin{lemma}
\label{lem.toggle}Let $U$ be a finite set with $U\neq\varnothing$. Then,%
\[
\sum_{F\subseteq U}\left(  -1\right)  ^{\left\vert F\right\vert }=0.
\]

\end{lemma}

\begin{proof}
This is an easy (and well-known) consequence of Lemma \ref{lem.sets.srvinv}.
It also follows from the well-known binomial identity $\sum_{k=0}^{n}\left(
-1\right)  ^{k}\dbinom{n}{k}=0$ that holds for any integer $n>0$.
\end{proof}

We can now easily derive Theorem \ref{thm.elser-Shade1} from Theorem
\ref{thm.elser-Shade0}:

\begin{proof}
[of Theorem \ref{thm.elser-Shade1}]We have%
\[
\sum_{F\subseteq E}\left(  -1\right)  ^{\left\vert F\right\vert }%
=\sum_{\substack{F\subseteq E;\\G\subseteq\operatorname*{Shade}F}}\left(
-1\right)  ^{\left\vert F\right\vert }+\underbrace{\sum_{\substack{F\subseteq
E;\\G\not \subseteq \operatorname*{Shade}F}}\left(  -1\right)  ^{\left\vert
F\right\vert }}_{\substack{=0\\\text{(by Theorem \ref{thm.elser-Shade0})}%
}}=\sum_{\substack{F\subseteq E;\\G\subseteq\operatorname*{Shade}F}}\left(
-1\right)  ^{\left\vert F\right\vert }.
\]
However, Lemma \ref{lem.toggle} (applied to $U=E$) shows that the left hand
side of this equality is $0$. Thus, so is the right hand side. This proves
Theorem \ref{thm.elser-Shade1}.
\end{proof}

\subsection{Proving Theorem \ref{thm.elser0}}

\begin{proof}
[of Theorem \ref{thm.elser0}]Theorem \ref{thm.elser0} follows by applying
Theorem \ref{thm.elser-Shade1} to $G=E$ (since a subset $F$ of $E$ satisfies
$E\subseteq\operatorname*{Shade}F$ if and only if it is pandemic).
\end{proof}

\section{Vertex infection and other variants}

In our study of graphs so far, we have barely ever mentioned vertices (even
though they are, of course, implicit in the notion of a path). Even though the
infection is spread from vertex to vertex, our sets so far have infected edges
(not vertices). One might thus wonder if there is also a vertex counterpart of
Theorem \ref{thm.elser0}. So let us define analogues of our notions for vertices:

If $F\subseteq V$, then an $F$\emph{-vertex-path} shall mean a path of
$\Gamma$ such that all vertices of the path except (possibly) for its two
endpoints belong to $F$. (Thus, if a path has only one edge or none, then it
automatically is an $F$-vertex-path.)

If $w\in V\setminus\left\{  v\right\}  $ is any vertex and $F\subseteq
V\setminus\left\{  v\right\}  $ is any subset, then we say that $F$
\emph{vertex-infects }$w$ if there exists an $F$-vertex-path from $v$ to $w$.
(This is always true when $w$ is $v$ or a neighbor of $v$.)

A subset $F\subseteq V\setminus\left\{  v\right\}  $ is said to be
\emph{vertex-pandemic} if it vertex-infects each vertex $w\in V\setminus
\left\{  v\right\}  $.

\begin{example}
Let $\Gamma$ be as in Example \ref{exa.example2}. Then, the path
$v\overset{1}{\longrightarrow}p\overset{2}{\longrightarrow}q$ is an
$F$-vertex-path for any subset $F\subseteq V$ that satisfies $p\in F$. The
subset $\left\{  p\right\}  $ of $V\setminus\left\{  v\right\}  $
vertex-infects each vertex (for example, $v\overset{1}{\longrightarrow
}p\overset{2}{\longrightarrow}q$ is a $\left\{  p\right\}  $-vertex-path from
$v$ to $q$, and $v\overset{4}{\longrightarrow}w$ is a $\left\{  p\right\}
$-vertex-path from $v$ to $w$), and thus is vertex-pandemic. The
vertex-pandemic subsets of $V\setminus\left\{  v\right\}  $ are the sets%
\[
\left\{  p\right\}  ,\ \ \left\{  w\right\}  ,\ \ \left\{  p,q\right\}
,\ \ \left\{  p,w\right\}  ,\ \ \left\{  q,w\right\}  ,\ \ \left\{
p,q,w\right\}  .
\]

\end{example}

We now have the following analogue of Theorem \ref{thm.elser0}:

\begin{theorem}
\label{thm.elser0-vert}Assume that $V\setminus\left\{  v\right\}
\neq\varnothing$. Then,%
\[
\sum_{\substack{F\subseteq V\setminus\left\{  v\right\}  \text{ is}%
\\\text{vertex-pandemic}}}\left(  -1\right)  ^{\left\vert F\right\vert }=0.
\]

\end{theorem}

\begin{proof}
With a few easy modifications, our above proof of Theorem \ref{thm.elser0} can
be repurposed as a proof of Theorem \ref{thm.elser0-vert}. Most importantly,
we need to replace the set $E$ by $V\setminus\left\{  v\right\}  $, and we
need to replace the words ``edge'',
``$F$-path'', ``infects'' and ``pandemic'' by ``vertex'',
``$F$-vertex-path'', ``vertex-infects'' and ``vertex-pandemic'', respectively.
\end{proof}

Another variant of Theorem \ref{thm.elser0} (and Theorem
\ref{thm.elser-Shade1} and Theorem \ref{thm.elser-Shade0}) is obtained by
replacing the undirected graph $\Gamma$ with a directed graph (while, of
course, replacing paths by directed paths). More generally, we can replace
$\Gamma$ by a ``hybrid'' graph with some
directed and some undirected edges.\footnote{We understand that a directed
edge still has two endpoints: its source and its target.} No changes are
required to the above proofs. Yet another variation can be obtained by
replacing ``endpoint'' by ``source'' (for directed edges). We cannot,
however, replace ``endpoint'' by ``target''.

\section{\label{sec.abstract}An abstract perspective}

Seeing how little graph theory we have used in proving Theorem
\ref{thm.elser0}, and how easily the same argument adapted to Theorem
\ref{thm.elser0-vert}, we get the impression that there might be some general
theory lurking behind it. What follows is an attempt at building this theory.

Most proofs in this section are omitted; some are outlined. In fact, they are
all sufficiently simple and straightforward that the reader should have little
trouble filling them in; alternatively, almost all of them can be found in the
detailed version of \cite{vershort}.

\subsection{Shade maps}

Let $\mathcal{P}\left(  E\right)  $ denote the power set of $E$. In Definition
\ref{def.Shade}, we have encoded the ``infects'' relation as a map $\operatorname*{Shade}%
:\mathcal{P}\left(  E\right)  \rightarrow\mathcal{P}\left(  E\right)  $
defined by $\operatorname*{Shade}F=\left\{  e\in E\ \mid\ F\text{ infects
}e\right\}  $. As we recall, Theorem \ref{thm.elser-Shade1} (a generalization
of Theorem \ref{thm.elser0}) states that
\begin{equation}
\sum_{\substack{F\subseteq E;\\G\subseteq\operatorname*{Shade}F}}\left(
-1\right)  ^{\left\vert F\right\vert }=0 \label{eq.Shade.goal}%
\end{equation}
for any $G\subseteq E$, under the assumption that $E\neq\varnothing$.

To generalize this, we forget about the graph $\Gamma$ and the map
$\operatorname*{Shade}$, and instead start with an \textbf{arbitrary} finite
set $E$. (This set $E$ corresponds to the set $E$ in Theorem \ref{thm.elser0}
and to the set $V\setminus\left\{  v\right\}  $ in Theorem
\ref{thm.elser0-vert}.) Let $\mathcal{P}\left(  E\right)  $ be the power set
of $E$. Let $\operatorname*{Shade}:\mathcal{P}\left(  E\right)  \rightarrow
\mathcal{P}\left(  E\right)  $ be an arbitrary map (meant to generalize the
map $\operatorname*{Shade}$ from the previous paragraph). We may now ask:

\begin{question}
What (combinatorial) properties must $\operatorname*{Shade}$ satisfy in order
for (\ref{eq.Shade.goal}) to hold for any $G\subseteq E$ under the assumption
that $E\neq\varnothing$ ?
\end{question}

A partial answer to this question can be given by analyzing our above proof of
Theorem \ref{thm.elser-Shade1} and extracting what was used:

\begin{definition}
\label{def.Shade.shademap}Let $E$ be a set. A \emph{shade map} on $E$ shall
mean a map $\operatorname*{Shade}:\mathcal{P}\left(  E\right)  \rightarrow
\mathcal{P}\left(  E\right)  $ that satisfies the following two axioms:

\begin{statement}
\textit{Axiom 1:} If $F\in\mathcal{P}\left(  E\right)  $ and $u\in
E\setminus\operatorname*{Shade}F$, then $\operatorname*{Shade}\left(
F\cup\left\{  u\right\}  \right)  =\operatorname*{Shade}F$.
\end{statement}

\begin{statement}
\textit{Axiom 2:} If $F\in\mathcal{P}\left(  E\right)  $ and $u\in
E\setminus\operatorname*{Shade}F$, then $\operatorname*{Shade}\left(
F\setminus\left\{  u\right\}  \right)  =\operatorname*{Shade}F$.
\end{statement}
\end{definition}

\begin{theorem}
\label{thm.elser.genE0}Let $E$ be a finite set. Let $\operatorname*{Shade}%
:\mathcal{P}\left(  E\right)  \rightarrow\mathcal{P}\left(  E\right)  $ be a
shade map on $E$.

Assume that $E\neq\varnothing$. Let $G$ be any subset of $E$. Then,%
\[
\sum_{\substack{F\subseteq E;\\G\subseteq\operatorname*{Shade}F}}\left(
-1\right)  ^{\left\vert F\right\vert }=0.
\]

\end{theorem}

\begin{proof}
Again, the proof is analogous to our above proof of Theorem
\ref{thm.elser-Shade1}. (This time, in the proof of Lemma \ref{lem.Shade-tog},
the equalities (\ref{eq.lem.Shade-tog.union}) and (\ref{eq.lem.Shade-tog.diff}%
) follow directly from Axiom 1 and Axiom 2, respectively.)
\end{proof}

How do shade maps relate to known concepts in the combinatorics of set
families (such as topologies, clutters, matroids, or submodular functions)?
Are they just one of these known concepts in disguise? We shall answer two
versions of this question in the following subsections. Specifically:

\begin{itemize}
\item In Subsection \ref{subsect.abstract.amatr}, we will show that
inclusion-reversing shade maps on $E$ (i.e., shade maps $\operatorname*{Shade}%
$ that satisfy $\operatorname*{Shade}B\subseteq\operatorname*{Shade}A$
whenever $A\subseteq B$) are in bijection with \emph{antimatroidal
quasi-closure operators} (a slight variant of antimatroids) on $E$.

\item In Subsection \ref{subsect.abstract.bip}, we will show that arbitrary
shade maps are in bijection with \emph{Boolean interval partitions} of
$\mathcal{P}\left(  E\right)  $ (that is, set partitions of $\mathcal{P}%
\left(  E\right)  $ into intervals of the Boolean lattice $\mathcal{P}\left(
E\right)  $).
\end{itemize}

Before we come to these characterizations, we shall however make a few
elementary remarks on shade maps.

First, we observe that Axioms 1 and 2 in Definition \ref{def.Shade.shademap}
can be weakened to the following statements:

\begin{statement}
\textit{Axiom 1':} If $F\in\mathcal{P}\left(  E\right)  $ and $u\in
E\setminus\operatorname*{Shade}F$, then $\operatorname*{Shade}\left(
F\cup\left\{  u\right\}  \right)  \subseteq\operatorname*{Shade}F$.
\end{statement}

\begin{statement}
\textit{Axiom 2':} If $F\in\mathcal{P}\left(  E\right)  $ and $u\in
E\setminus\operatorname*{Shade}F$, then $\operatorname*{Shade}\left(
F\setminus\left\{  u\right\}  \right)  \subseteq\operatorname*{Shade}F$.
\end{statement}

Axiom 1' is weaker than Axiom 1, and likewise Axiom 2' is weaker than Axiom 2.
However, Axioms 1' and 2' combined are equivalent to Axioms 1 and 2 combined:

\begin{proposition}
\label{prop.Shade.shade-2}Let $E$ be a set. Let $\operatorname*{Shade}%
:\mathcal{P}\left(  E\right)  \rightarrow\mathcal{P}\left(  E\right)  $ be any
map. Then, $\operatorname*{Shade}$ is a shade map on $E$ if and only if
$\operatorname*{Shade}$ satisfies the two Axioms 1' and 2' stated above.
\end{proposition}

Axioms 1 and 2 can also be combined into one common axiom:

\begin{statement}
\textit{Axiom 3:} If $F\in\mathcal{P}\left(  E\right)  $ and $u\in E\setminus
F$, then we have $\operatorname*{Shade}F=\operatorname*{Shade}\left(
F\cup\left\{  u\right\}  \right)  $ or $u\in\left(  \operatorname*{Shade}%
F\right)  \cap\operatorname*{Shade}\left(  F\cup\left\{  u\right\}  \right)  $.
\end{statement}

\begin{proposition}
\label{prop.Shade.shade-3}Let $E$ be a set. Let $\operatorname*{Shade}%
:\mathcal{P}\left(  E\right)  \rightarrow\mathcal{P}\left(  E\right)  $ be any
map. Then, $\operatorname*{Shade}$ is a shade map on $E$ if and only if
$\operatorname*{Shade}$ satisfies Axiom 3.
\end{proposition}

We will soon see some examples. First, let us introduce two more basic
concepts that will help clarify these examples:

\begin{definition}
\label{def.inclusion-pres}Let $E$ be a set. Let $\operatorname*{Shade}%
:\mathcal{P}\left(  E\right)  \rightarrow\mathcal{P}\left(  E\right)  $ be any
map (not necessarily a shade map).

\textbf{(a)} We say that $\operatorname*{Shade}$ is
\emph{inclusion-preserving} if it satisfies the following property: If $A$ and
$B$ are two subsets of $E$ such that $A\subseteq B$, then
$\operatorname*{Shade}A\subseteq\operatorname*{Shade}B$.

\textbf{(b)} We say that $\operatorname*{Shade}$ is \emph{inclusion-reversing}
if it satisfies the following property: If $A$ and $B$ are two subsets of $E$
such that $A\subseteq B$, then $\operatorname*{Shade}B\subseteq
\operatorname*{Shade}A$.
\end{definition}

For instance, the map $\operatorname*{Shade}$ from Definition \ref{def.Shade}
is inclusion-preserving (because of Lemma \ref{lem.Shade-monoton}) and is a
shade map (by Lemma \ref{lem.Shade-tog}). The same holds for the analogue of
the map $\operatorname*{Shade}$ that uses vertex-infection instead of
infection. We will soon see some inclusion-reversing shade maps, and it is not
hard to construct shade maps that are neither inclusion-preserving nor inclusion-reversing.

Let us observe that there is a simple bijection between inclusion-preserving
and inclusion-reversing maps, and this bijection preserves shadeness:

\begin{proposition}
\label{prop.Shade.dual}Let $E$ be a set. Let $\operatorname*{Shade}%
:\mathcal{P}\left(  E\right)  \rightarrow\mathcal{P}\left(  E\right)  $ be any
map (not necessarily a shade map). Let $\operatorname*{Shade}\nolimits^{\prime
}:\mathcal{P}\left(  E\right)  \rightarrow\mathcal{P}\left(  E\right)  $ be
the map that sends each $F\in\mathcal{P}\left(  E\right)  $ to
$\operatorname*{Shade}\left(  E\setminus F\right)  \in\mathcal{P}\left(
E\right)  $. Then:

\textbf{(a)} The map $\operatorname*{Shade}$ is inclusion-preserving if and
only if the map $\operatorname*{Shade}\nolimits^{\prime}$ is inclusion-reversing.

\textbf{(b)} The map $\operatorname*{Shade}$ is a shade map if and only if the
map $\operatorname*{Shade}\nolimits^{\prime}$ is a shade map.
\end{proposition}

\begin{definition}
Let $E$, $\operatorname*{Shade}$ and $\operatorname*{Shade}\nolimits^{\prime}$
be as in Proposition \ref{prop.Shade.dual}. We then say that the map
$\operatorname*{Shade}\nolimits^{\prime}$ is \emph{dual} to
$\operatorname*{Shade}$.
\end{definition}

\subsection{Some examples of shade maps}

As we already mentioned, Lemma \ref{lem.Shade-tog} and its analogue for
vertex-infection provide two examples of inclusion-preserving shade maps
$\operatorname*{Shade}$. An example of an inclusion-reversing shade map comes
from the theory of posets:

\begin{example}
\label{exa.shade.low}Let $E$ be a poset. For any $F\subseteq E$, we define%
\[
F_{\downarrow}=\left\{  e\in E\ \mid\ \text{there exists an }f\in F\text{ with
}e<f\right\}
\]
and%
\[
\operatorname*{Shade}F=E\setminus F_{\downarrow}.
\]
Then, this map $\operatorname*{Shade}:\mathcal{P}\left(  E\right)
\rightarrow\mathcal{P}\left(  E\right)  $ is an inclusion-reversing shade map.
\end{example}

Another example of a shade map comes from discrete geometry:

\begin{example}
\label{exa.shade.convex}Let $A$ be an affine space over $\mathbb{R}$. If $S$
is a finite subset of $A$, then a \emph{nontrivial convex combination} of $S$
will mean a point of the form $\sum_{s\in S}\lambda_{s}s\in A$, where the
coefficients $\lambda_{s}$ are nonnegative reals smaller than $1$ and
satisfying $\sum_{s\in S}\lambda_{s}=1$.

Fix a finite subset $E$ of $A$. For any $F\subseteq E$, we define%
\[
\operatorname*{Shade}F=\left\{  e\in E\ \mid\ e\text{ is \textbf{not} a
nontrivial convex combination of }F\right\}  .
\]
Then, this map $\operatorname*{Shade}:\mathcal{P}\left(  E\right)
\rightarrow\mathcal{P}\left(  E\right)  $ is an inclusion-reversing shade map.
\end{example}

As a contrast to Example \ref{exa.shade.convex}, let us mention a
not-quite-example (satisfying only one of the two axioms in Theorem
\ref{thm.elser.genE0}):

\begin{example}
\label{exa.shade.cone}Let $V$ be a vector space over $\mathbb{R}$. If $S$ is a
finite subset of $V$, then a \emph{nontrivial conic combination} of $S$ will
mean a vector of the form $\sum_{s\in S}\lambda_{s}s\in V$, where the
coefficients $\lambda_{s}$ are nonnegative reals with the property that at
least two elements $s\in S$ satisfy $\lambda_{s}>0$.

Fix a finite subset $E$ of $V$. For any $F\subseteq E$, we define%
\[
\operatorname*{Shade}F=\left\{  e\in E\ \mid\ e\text{ is \textbf{not} a
nontrivial conic combination of }F\right\}  .
\]
It can be shown that this map $\operatorname*{Shade}:\mathcal{P}\left(
E\right)  \rightarrow\mathcal{P}\left(  E\right)  $ satisfies Axiom 1 in
Definition \ref{def.Shade.shademap}. In general, it does not satisfy Axiom 2.
Thus, it is not a shade map in general.
\end{example}

\subsection{\label{subsect.abstract.amatr}Antimatroids and inclusion-reversing
shade maps}

Examples \ref{exa.shade.low} and \ref{exa.shade.convex} are instances of a
general class of examples: shade maps coming from \emph{antimatroids}. Not
unlike matroids, antimatroids are a combinatorial concept with many equivalent
avatars (see, e.g., \cite[Chapter III]{KoLoSc91}). Here we shall view them
through one of these avatars: that of \emph{antimatroidal quasi-closure
operators} (roughly equivalent to \emph{convex geometries}). We begin by
defining the notions we need:

\begin{definition}
\label{def.antimat}Let $E$ be any set.

\textbf{(a)} A \emph{quasi-closure operator} on $E$ means a map $\tau
:\mathcal{P}\left(  E\right)  \rightarrow\mathcal{P}\left(  E\right)  $ with
the following properties:

\begin{enumerate}
\item We have $A\subseteq\tau\left(  A\right)  $ for any $A\subseteq E$.

\item If $A$ and $B$ are two subsets of $E$ satisfying $A\subseteq B$, then
$\tau\left(  A\right)  \subseteq\tau\left(  B\right)  $.

\item We have $\tau\left(  \tau\left(  A\right)  \right)  =\tau\left(
A\right)  $ for any $A\subseteq E$.
\end{enumerate}

\textbf{(b)} A quasi-closure operator $\tau$ on $E$ is said to be
\emph{antimatroidal} if it has the following additional property:

\begin{enumerate}
\item[4.] If $X$ is a subset of $E$, and if $y$ and $z$ are two distinct
elements of $E\setminus\tau\left(  X\right)  $ satisfying $z\in\tau\left(
X\cup\left\{  y\right\}  \right)  $, then $y\notin\tau\left(  X\cup\left\{
z\right\}  \right)  $.
\end{enumerate}

\textbf{(c)} A \emph{closure operator} on $E$ means a quasi-closure operator
$\tau$ on $E$ that satisfies $\tau\left(  \varnothing\right)  =\varnothing$.

\textbf{(d)} If $\tau$ is an antimatroidal closure operator on $E$, then the
pair $\left(  E,\tau\right)  $ is called a \emph{convex geometry}.
\end{definition}

Here are some examples of antimatroidal quasi-closure operators:

\begin{example}
\label{exa.antimat.low}Let $E$ be a poset. For any $F\subseteq E$, we define%
\[
\tau\left(  F\right)  =\left\{  e\in E\ \mid\ \text{there exists an }f\in
F\text{ with }e\leq f\right\}  .
\]
Then, $\tau$ is an antimatroidal closure operator on $E$. (This example is the
``downset alignment'' from \cite[\S 3,
Example II]{EdeJam85}, and is equivalent to the ``poset
antimatroid'' from \cite[\S III.2.3]{KoLoSc91}.)
\end{example}

\begin{example}
\label{exa.antimat.interval}Let $E$ be a poset. For any $F\subseteq E$, we
define%
\[
\tau\left(  F\right)  =\left\{  e\in E\ \mid\ \text{there exist }f\in F\text{
and }g\in F\text{ with }g\leq e\leq f\right\}  .
\]
Then, $\tau$ is an antimatroidal closure operator on $E$. (This example is the
``order convex alignment'' from \cite[\S 3,
Example II]{EdeJam85}, and is the ``double shelling of a
poset'' example from \cite[\S III.2.4]{KoLoSc91}.)
\end{example}

\begin{example}
\label{exa.antimat.convex}Let $A$ be an affine space over $\mathbb{R}$. If $S$
is a finite subset of $A$, then a \emph{convex combination} of $S$ will mean a
point of the form $\sum_{s\in S}\lambda_{s}s\in A$, where the coefficients
$\lambda_{s}$ are nonnegative reals satisfying $\sum_{s\in S}\lambda_{s}=1$.

Fix a finite subset $E$ of $A$. For any $F\subseteq E$, we define%
\[
\tau\left(  F\right)  =\left\{  e\in E\ \mid\ e\text{ is a convex combination
of }F\right\}  .
\]
Then, $\tau$ is an antimatroidal closure operator on $E$. (This example is
\cite[\S 3, Example I]{EdeJam85}; it gave the name ``convex
geometry'' to the notion defined in Definition
\ref{def.antimat} \textbf{(d)}.)
\end{example}

\begin{example}
\label{exa.antimat.linesearch}Let $\Gamma$ be any graph with edge set $E$. Fix
a vertex $v$ of $\Gamma$. We say that a subset $F\subseteq E$ \emph{blocks} an
edge $e\in E$ if each path of $\Gamma$ that contains $v$ and $e$ must contain
at least one edge of $F$. (In particular, this is automatically the case when
$e\in F$.) For each $F\subseteq E$, we define%
\[
\tau\left(  F\right)  =\left\{  e\in E\ \mid\ F\text{ blocks }e\right\}  .
\]
Then, $\tau$ is an antimatroidal quasi-closure operator on $E$. (This example
is the ``line-search antimatroid'' from
\cite[\S III.2.11]{KoLoSc91}.) If $\Gamma$ is connected, then $\tau$ is
actually a closure operator.
\end{example}

Further examples of antimatroidal closure operators can be found in
\cite[\S III.2]{KoLoSc91} and \cite[\S 3]{EdeJam85}.

We shall be dealing with quasi-closure operators rather than closure operators
most of the time. However, since the latter concept is somewhat more
widespread, let us comment on the connection between the two. Roughly
speaking, the relation between quasi-closure and closure operators is
comparable to the relation between semigroups and monoids, or between
nonunital rings and unital rings, or (perhaps the best analogue) between
simplicial complexes in general and simplicial complexes without ghost
vertices (i.e., simplicial complexes for which every element of the ground set
is a dimension-$0$ face). More concretely, specifying a quasi-closure operator
on a set $E$ is tantamount to specifying a subset of $E$ and a closure
operator on this subset. To wit:

\begin{proposition}
\label{prop.antimat.cl-vs-qcl}Let $E$ be a set. Let $L$ be a subset of $E$.

\textbf{(a)} Then, there is a bijection%
\begin{align*}
&  \text{from }\left\{  \text{quasi-closure operators }\tau\text{ on }E\text{
satisfying }\tau\left(  \varnothing\right)  =L\right\} \\
&  \text{to }\left\{  \text{closure operators }\sigma\text{ on }E\setminus
L\right\}
\end{align*}
that is defined as follows: It sends each quasi-closure operator $\tau$ to the
closure operator $\sigma$ that sends each $F\subseteq E\setminus L$ to
$\tau\left(  F\right)  \setminus L$.

\textbf{(b)} This bijection restricts to a bijection%
\begin{align*}
&  \text{from }\left\{  \text{antimatroidal quasi-closure operators }%
\tau\text{ on }E\text{ satisfying }\tau\left(  \varnothing\right)  =L\right\}
\\
&  \text{to }\left\{  \text{antimatroidal closure operators }\sigma\text{ on
}E\setminus L\right\}  .
\end{align*}

\end{proposition}

Now, we claim the following:

\begin{theorem}
\label{thm.antimat.Shade}Let $E$ be a set. Let $\tau:\mathcal{P}\left(
E\right)  \rightarrow\mathcal{P}\left(  E\right)  $ be an antimatroidal
quasi-closure operator on $E$. For any $F\subseteq E$, we define%
\begin{equation}
\operatorname*{Shade}F=\left\{  e\in E\ \mid\ e\notin\tau\left(
F\setminus\left\{  e\right\}  \right)  \right\}  .
\label{eq.thm.antimat.Shade.ShadeF=}%
\end{equation}
Then, this map $\operatorname*{Shade}:\mathcal{P}\left(  E\right)
\rightarrow\mathcal{P}\left(  E\right)  $ is an inclusion-reversing shade map.
\end{theorem}

Theorem \ref{thm.antimat.Shade} generalizes Examples \ref{exa.shade.low} and
\ref{exa.shade.convex}. Indeed, the latter two examples are obtained by
applying Theorem \ref{thm.antimat.Shade} to the settings of Examples
\ref{exa.antimat.low} and \ref{exa.antimat.convex}, respectively. Less
directly, Lemma \ref{lem.Shade-tog} and its vertex-infection analogue are
particular cases of Theorem \ref{thm.antimat.Shade} as well (even though they
involve shade maps that are inclusion-preserving rather than
inclusion-reversing). Indeed, if we apply Theorem \ref{thm.antimat.Shade} to
the setting of Example \ref{exa.antimat.linesearch}, then we obtain the claim
of Lemma \ref{lem.Shade-tog} with $\operatorname*{Shade}F$ replaced by
$\operatorname*{Shade}\left(  E\setminus F\right)  $; this is easily seen to
be equivalent to Lemma \ref{lem.Shade-tog} (by the duality stated in
Proposition \ref{prop.Shade.dual}).

Our proof of Theorem \ref{thm.antimat.Shade} will rely on two easy lemmas:

\begin{lemma}
\label{lem.antimat.clos-add}Let $E$ be a set. Let $\tau$ be a quasi-closure
operator on $E$. Let $X$ be a subset of $E$, and let $z\in\tau\left(
X\right)  $. Then, $\tau\left(  X\cup\left\{  z\right\}  \right)  =\tau\left(
X\right)  $.
\end{lemma}

\begin{lemma}
\label{lem.antimat.nicer}Let $E$ be a set. Let $\tau$ be an antimatroidal
quasi-closure operator on $E$. Let $X$ be a subset of $E$, and let $y$ and $z$
be two distinct elements of $E$ satisfying $z\in\tau\left(  X\cup\left\{
y\right\}  \right)  $ and $y\in\tau\left(  X\cup\left\{  z\right\}  \right)
$. Then, $y\in\tau\left(  X\right)  $.
\end{lemma}

Note that Lemma \ref{lem.antimat.nicer} has a converse: If $\tau$ is a
quasi-closure operator on $E$ satisfying the claim of Lemma
\ref{lem.antimat.nicer}, then $\tau$ is antimatroidal. This is easy to see but
will not be used in what follows.

\begin{proof}
[of Theorem \ref{thm.antimat.Shade}]We shall prove the following three statements:

\begin{statement}
\textit{Statement 0:} If $A$ and $B$ are two subsets of $E$ such that
$A\subseteq B$, then $\operatorname*{Shade}B\subseteq\operatorname*{Shade}A$.
\end{statement}

\begin{statement}
\textit{Statement 1:} If $F\in\mathcal{P}\left(  E\right)  $ and $u\in
E\setminus\operatorname*{Shade}F$, then $\operatorname*{Shade}\left(
F\cup\left\{  u\right\}  \right)  =\operatorname*{Shade}F$.
\end{statement}

\begin{statement}
\textit{Statement 2:} If $F\in\mathcal{P}\left(  E\right)  $ and $u\in
E\setminus\operatorname*{Shade}F$, then $\operatorname*{Shade}\left(
F\setminus\left\{  u\right\}  \right)  =\operatorname*{Shade}F$.
\end{statement}

Our proofs of these three statements will use Property 2 in Definition
\ref{def.antimat} \textbf{(a)} (which we shall just refer to as
``Property 2'').

[\textit{Proof of Statement 0:} Let $A$ and $B$ be two subsets of $E$ such
that $A\subseteq B$. We must prove that $\operatorname*{Shade}B\subseteq
\operatorname*{Shade}A$.

Let $u\in\operatorname*{Shade}B$. Thus, $u\in E$ and $u\notin\tau\left(
B\setminus\left\{  u\right\}  \right)  $ (by the definition of
$\operatorname*{Shade}B$). However, $A\setminus\left\{  u\right\}  \subseteq
B\setminus\left\{  u\right\}  $ (since $A\subseteq B$) and thus $\tau\left(
A\setminus\left\{  u\right\}  \right)  \subseteq\tau\left(  B\setminus\left\{
u\right\}  \right)  $ (by Property 2). Hence, from $u\notin\tau\left(
B\setminus\left\{  u\right\}  \right)  $, we obtain $u\notin\tau\left(
A\setminus\left\{  u\right\}  \right)  $. Therefore, $u\in
\operatorname*{Shade}A$ (by the definition of $\operatorname*{Shade}A$).

Since we have shown this for each $u\in\operatorname*{Shade}B$, we thus obtain
$\operatorname*{Shade}B\subseteq\operatorname*{Shade}A$. This proves Statement 0.]

[\textit{Proof of Statement 2:} Let $F\in\mathcal{P}\left(  E\right)  $ and
$u\in E\setminus\operatorname*{Shade}F$. We must prove that
$\operatorname*{Shade}\left(  F\setminus\left\{  u\right\}  \right)
=\operatorname*{Shade}F$.

We have $u\in E\setminus\operatorname*{Shade}F$, so that $u\notin%
\operatorname*{Shade}F$. In other words, $u\in E$ and $u\in\tau\left(
F\setminus\left\{  u\right\}  \right)  $ (by the definition of
$\operatorname*{Shade}F$).

We have $F\setminus\left\{  u\right\}  \subseteq F$ and thus
$\operatorname*{Shade}F\subseteq\operatorname*{Shade}\left(  F\setminus
\left\{  u\right\}  \right)  $ (by Statement 0).

Now, let $v\in\operatorname*{Shade}\left(  F\setminus\left\{  u\right\}
\right)  $. We shall prove that $v\in\operatorname*{Shade}F$.

Indeed, assume the contrary. Hence, $v\notin\operatorname*{Shade}F$. In other
words, $v\in E$ and $v\in\tau\left(  F\setminus\left\{  v\right\}  \right)  $
(by the definition of $\operatorname*{Shade}F$).

Let $X=\left(  F\setminus\left\{  u\right\}  \right)  \setminus\left\{
v\right\}  $. Then, $F\setminus\left\{  u\right\}  \subseteq X\cup\left\{
v\right\}  $ and therefore $\tau\left(  F\setminus\left\{  u\right\}  \right)
\subseteq\tau\left(  X\cup\left\{  v\right\}  \right)  $ (by Property 2).
Hence, $u\in\tau\left(  F\setminus\left\{  u\right\}  \right)  \subseteq
\tau\left(  X\cup\left\{  v\right\}  \right)  $.

Also, from $X=\left(  F\setminus\left\{  u\right\}  \right)  \setminus\left\{
v\right\}  =\left(  F\setminus\left\{  v\right\}  \right)  \setminus\left\{
u\right\}  $, we obtain $F\setminus\left\{  v\right\}  \subseteq X\cup\left\{
u\right\}  $ and therefore $\tau\left(  F\setminus\left\{  v\right\}  \right)
\subseteq\tau\left(  X\cup\left\{  u\right\}  \right)  $ (by Property 2).
Hence, $v\in\tau\left(  F\setminus\left\{  v\right\}  \right)  \subseteq
\tau\left(  X\cup\left\{  u\right\}  \right)  $.

However, from $v\in\operatorname*{Shade}\left(  F\setminus\left\{  u\right\}
\right)  $, we obtain $v\notin\tau\left(  \left(  F\setminus\left\{
u\right\}  \right)  \setminus\left\{  v\right\}  \right)  $ (by the definition
of $\operatorname*{Shade}\left(  F\setminus\left\{  u\right\}  \right)  $). In
other words, $v\notin\tau\left(  X\right)  $ (since $X=\left(  F\setminus
\left\{  u\right\}  \right)  \setminus\left\{  v\right\}  $). However, if $v$
and $u$ were distinct, then Lemma \ref{lem.antimat.nicer} (applied to $y=v$
and $z=u$) would yield $v\in\tau\left(  X\right)  $ (since $v\in\tau\left(
X\cup\left\{  u\right\}  \right)  $ and $u\in\tau\left(  X\cup\left\{
v\right\}  \right)  $), which would contradict $v\notin\tau\left(  X\right)
$. Thus, $v$ and $u$ cannot be distinct. In other words, $v=u$. Now,
$X=\left(  F\setminus\left\{  u\right\}  \right)  \setminus\left\{  v\right\}
=F\setminus\left\{  u\right\}  $ (since $v=u$). Therefore, $u\in\tau\left(
F\setminus\left\{  u\right\}  \right)  $ rewrites as $u\in\tau\left(
X\right)  $. Hence, $v=u\in\tau\left(  X\right)  $; but this contradicts
$v\notin\tau\left(  X\right)  $.

This contradiction shows that our assumption was false; hence, we conclude
that $v\in\operatorname*{Shade}F$. Since we have proved this for each
$v\in\operatorname*{Shade}\left(  F\setminus\left\{  u\right\}  \right)  $, we
thus obtain $\operatorname*{Shade}\left(  F\setminus\left\{  u\right\}
\right)  \subseteq\operatorname*{Shade}F$. Combining this with
$\operatorname*{Shade}F\subseteq\operatorname*{Shade}\left(  F\setminus
\left\{  u\right\}  \right)  $, we obtain $\operatorname*{Shade}\left(
F\setminus\left\{  u\right\}  \right)  =\operatorname*{Shade}F$. This proves
Statement 2.]

[\textit{Proof of Statement 1:} Let $F\in\mathcal{P}\left(  E\right)  $ and
$u\in E\setminus\operatorname*{Shade}F$. We must prove that
$\operatorname*{Shade}\left(  F\cup\left\{  u\right\}  \right)
=\operatorname*{Shade}F$. If $u\in F$, then this is obvious (since
$F\cup\left\{  u\right\}  =F$ in this case). Thus, we WLOG assume that
$u\notin F$. Hence, $\left(  F\cup\left\{  u\right\}  \right)  \setminus
\left\{  u\right\}  =F$.

We have $F\subseteq F\cup\left\{  u\right\}  $ and thus $\operatorname*{Shade}%
\left(  F\cup\left\{  u\right\}  \right)  \subseteq\operatorname*{Shade}F$ (by
Statement 0).

Now, $u\in E\setminus\operatorname*{Shade}F\subseteq E\setminus
\operatorname*{Shade}\left(  F\cup\left\{  u\right\}  \right)  $ (since
$\operatorname*{Shade}\left(  F\cup\left\{  u\right\}  \right)  \subseteq
\operatorname*{Shade}F$). Hence, Statement 2 (applied to $F\cup\left\{
u\right\}  $ instead of $F$) yields $\operatorname*{Shade}%
F=\operatorname*{Shade}\left(  F\cup\left\{  u\right\}  \right)  $ (since
$\left(  F\cup\left\{  u\right\}  \right)  \setminus\left\{  u\right\}  =F$).
This proves Statement 1.]

Now, we have proved Statements 1 and 2. Thus, the map $\operatorname*{Shade}%
:\mathcal{P}\left(  E\right)  \rightarrow\mathcal{P}\left(  E\right)  $ is a
shade map. Moreover, this map is inclusion-reversing (by Statement 0). Thus,
Theorem \ref{thm.antimat.Shade} is proved.
\end{proof}

We note that the quasi-closure operator $\tau$ in Theorem
\ref{thm.antimat.Shade} can be reconstructed from the map
$\operatorname*{Shade}$. This does not even require $\tau$ to be
antimatroidal; the following holds for any quasi-closure operator:

\begin{proposition}
\label{prop.antimat.recover-tau}Let $E$ be a set. Let $\tau:\mathcal{P}\left(
E\right)  \rightarrow\mathcal{P}\left(  E\right)  $ be a quasi-closure
operator on $E$. For any $F\subseteq E$, we define%
\[
\operatorname*{Shade}F=\left\{  e\in E\ \mid\ e\notin\tau\left(
F\setminus\left\{  e\right\}  \right)  \right\}  .
\]
Then, each $F\subseteq E$ satisfies%
\begin{equation}
\tau\left(  F\right)  =F\cup\left(  E\setminus\operatorname*{Shade}F\right)  .
\label{eq.prop.antimat.recover-tau.tauF=}%
\end{equation}

\end{proposition}

It turns out that if one applies the formula
(\ref{eq.prop.antimat.recover-tau.tauF=}) to an inclusion-reversing shade map
$\operatorname*{Shade}$, then the resulting map $\tau$ is an antimatroidal
quasi-closure operator, at least when $E$ is finite. In fact, we have the following:

\begin{proposition}
\label{prop.antimat.tau-from-Shade}Let $E$ be a finite set. Let
$\operatorname*{Shade}:\mathcal{P}\left(  E\right)  \rightarrow\mathcal{P}%
\left(  E\right)  $ be an inclusion-reversing shade map. Define a map
$\tau:\mathcal{P}\left(  E\right)  \rightarrow\mathcal{P}\left(  E\right)  $
by setting%
\begin{equation}
\tau\left(  F\right)  =F\cup\left(  E\setminus\operatorname*{Shade}F\right)
\label{eq.prop.antimat.tau-from-Shade.tauF=}%
\end{equation}
for each $F\subseteq E$. Then, $\tau$ is an antimatroidal quasi-closure
operator on $E$.
\end{proposition}

The proof of this proposition rests on the following lemma:

\begin{lemma}
\label{lem.Shade.FuG}Let $E$ be a set. Let $\operatorname*{Shade}%
:\mathcal{P}\left(  E\right)  \rightarrow\mathcal{P}\left(  E\right)  $ be a
shade map. Let $A$ and $B$ be two subsets of $E$ such that $B$ is finite and
$B\cap\operatorname*{Shade}A=\varnothing$. Then, $\operatorname*{Shade}\left(
A\cup B\right)  =\operatorname*{Shade}A$.
\end{lemma}

Proposition \ref{prop.antimat.recover-tau} has a (sort of) converse:

\begin{proposition}
\label{prop.antimat.recover-Shade}Let $E$ be a set. Let $\operatorname*{Shade}%
:\mathcal{P}\left(  E\right)  \rightarrow\mathcal{P}\left(  E\right)  $ be an
inclusion-reversing shade map. For any $F\subseteq E$, we define%
\[
\tau\left(  F\right)  =F\cup\left(  E\setminus\operatorname*{Shade}F\right)
.
\]
Then, each $F\subseteq E$ satisfies%
\begin{equation}
\operatorname*{Shade}F=\left\{  e\in E\ \mid\ e\notin\tau\left(
F\setminus\left\{  e\right\}  \right)  \right\}  .
\label{eq.prop.antimat.recover-Shade.ShadeF=}%
\end{equation}

\end{proposition}

Combining many of the results in this section, we obtain the following
description of inclusion-reversing shade maps:

\begin{theorem}
\label{thm.antimat.Shade.inc-rev-class}Let $E$ be a finite set. Then, there is
a bijection
\begin{align*}
&  \text{from the set }\left\{  \text{inclusion-reversing shade maps
}\operatorname*{Shade}:\mathcal{P}\left(  E\right)  \rightarrow\mathcal{P}%
\left(  E\right)  \right\} \\
&  \text{to the set }\left\{  \text{antimatroidal quasi-closure operators
}\tau:\mathcal{P}\left(  E\right)  \rightarrow\mathcal{P}\left(  E\right)
\right\}  .
\end{align*}
It sends each map $\operatorname*{Shade}$ to the map $\tau$ defined by
(\ref{eq.prop.antimat.tau-from-Shade.tauF=}). Its inverse map sends each map
$\tau$ to the map $\operatorname*{Shade}$ defined by
(\ref{eq.thm.antimat.Shade.ShadeF=}).
\end{theorem}

Theorem \ref{thm.antimat.Shade.inc-rev-class} classifies
\textbf{inclusion-reversing} shade maps in terms of antimatroidal
quasi-closure operators\footnote{In the parlance of matroid theorists, it
shows that inclusion-reversing shade maps are \emph{cryptomorphic} to
antimatroidal quasi-closure operators.}. The latter can in turn be described
in terms of antimatroidal closure operators (by Proposition
\ref{prop.antimat.cl-vs-qcl}), i.e., in terms of antimatroids. Thus,
inclusion-reversing shade maps ``boil down''
to antimatroids. The same can be said of \textbf{inclusion-preserving} shade
maps (because Proposition \ref{prop.Shade.dual} establishes a bijection
between them and the inclusion-reversing ones). In the next subsection, we
shall classify \textbf{arbitrary} shade maps in terms of what we will call
\emph{Boolean interval partitions}.

\subsection{\label{subsect.abstract.bip}Boolean interval partitions and
arbitrary shade maps}

Let us first define Boolean interval partitions:

\begin{definition}
\label{def.bip.bip}Let $E$ be a set.

\textbf{(a)} If $U$ and $V$ are two subsets of $E$ satisfying $U\subseteq V$,
then $\left[  U,V\right]  $ shall denote the subset%
\[
\left\{  I\in\mathcal{P}\left(  E\right)  \ \mid\ U\subseteq I\subseteq
V\right\}
\]
of $\mathcal{P}\left(  E\right)  $. This is the set of all subsets of $E$ that
lie between $U$ and $V$ (meaning that they contain $U$ as a subset, but in
turn are contained in $V$ as subsets).

\textbf{(b)} A \emph{Boolean interval} of $\mathcal{P}\left(  E\right)  $
shall mean a subset of $\mathcal{P}\left(  E\right)  $ that has the form
$\left[  U,V\right]  $ for two subsets $U$ and $V$ of $E$ satisfying
$U\subseteq V$. Note that each Boolean interval $\left[  U,V\right]  $ of
$\mathcal{P}\left(  E\right)  $ is nonempty (as it contains $U$ and $V$), and
the two subsets $U$ and $V$ can easily be reconstructed from it (namely, $U$
is the intersection of all $I\in\left[  U,V\right]  $, whereas $V$ is the
union of all $I\in\left[  U,V\right]  $).

\textbf{(c)} A \emph{Boolean interval partition} of $\mathcal{P}\left(
E\right)  $ means a set of pairwise disjoint Boolean intervals of
$\mathcal{P}\left(  E\right)  $ whose union is $\mathcal{P}\left(  E\right)  $.

\textbf{(d)} If $\mathbf{P}$ is a Boolean interval partition of $\mathcal{P}%
\left(  E\right)  $, then the elements of $\mathbf{P}$ (that is, the Boolean
intervals that belong to $\mathbf{P}$) are called the \emph{blocks} of
$\mathbf{P}$.
\end{definition}

\begin{example}
\label{exa.bip.bip.123}For this example, let $E=\left\{  1,2,3\right\}  $. We
shall use the shorthand $i_{1}i_{2}\cdots i_{k}$ for a subset $\left\{
i_{1},i_{2},\ldots,i_{k}\right\}  $ of $E$. (For example, $13$ means the
subset $\left\{  1,3\right\}  $.)

\textbf{(a)} We have
\[
\left[  1,\ 123\right]  =\left\{  I\in\mathcal{P}\left(  123\right)
\ \mid\ 1\subseteq I\subseteq123\right\}  =\left\{  1,\ 12,\ 13,\ 123\right\}
\]
and%
\[
\left[  1,\ 13\right]  =\left\{  I\in\mathcal{P}\left(  123\right)
\ \mid\ 1\subseteq I\subseteq13\right\}  =\left\{  1,\ 13\right\}  .
\]

\textbf{(b)} There are $3^{3}=27$ Boolean intervals of $\mathcal{P}\left(
E\right)  $. (More generally, if $E$ is an $n$-element set, then there are
$3^{n}$ Boolean intervals of $\mathcal{P}\left(  E\right)  $.)

\textbf{(c)} Here is one of many Boolean interval partitions of $\mathcal{P}%
\left(  E\right)  $ (where $E$ is still $\left\{  1,2,3\right\}  $):%
\[
\left\{  \underbrace{\left\{  \varnothing\right\}  }_{=\left[  \varnothing
,\ \varnothing\right]  },\ \ \underbrace{\left\{  1,\ 13\right\}  }_{=\left[
1,\ 13\right]  },\ \ \underbrace{\left\{  3\right\}  }_{=\left[  3,\ 3\right]
},\ \ \underbrace{\left\{  2,\ 12,\ 23,\ 123\right\}  }_{=\left[
2,\ 123\right]  }\right\}  .
\]
Here is another:%
\[
\left\{  \underbrace{\left\{  \varnothing,\ 1\right\}  }_{=\left[
\varnothing,\ 1\right]  },\ \ \underbrace{\left\{  3,\ 13\right\}  }_{=\left[
3,\ 13\right]  },\ \ \underbrace{\left\{  2,\ 23\right\}  }_{=\left[
2,\ 23\right]  },\ \ \underbrace{\left\{  12\right\}  }_{=\left[
12,\ 12\right]  },\ \ \underbrace{\left\{  123\right\}  }_{=\left[
123,\ 123\right]  }\right\}  .
\]
The former has four blocks; the latter has five.
\end{example}

Here are two ways to think of Boolean interval partitions of $\mathcal{P}%
\left(  E\right)  $:

\begin{itemize}
\item The following is just a slick restatement of Definition
\ref{def.bip.bip} \textbf{(c)} using standard combinatorial lingo: A Boolean
interval partition of $\mathcal{P}\left(  E\right)  $ is a set partition of
the Boolean lattice $\mathcal{P}\left(  E\right)  $ into intervals.

\item It is well-known that the set partitions of a given set are in a
canonical bijection with the equivalence relations on this set. In light of
this, the Boolean interval partitions of $\mathcal{P}\left(  E\right)  $ can
be viewed as the equivalence relations on $\mathcal{P}\left(  E\right)  $
whose equivalence classes are Boolean intervals. In other words, they can be
viewed as the equivalence relations $\sim$ on $\mathcal{P}\left(  E\right)  $
satisfying the axiom ``if $U,V,I\in\mathcal{P}\left(
E\right)  $ satisfy $U\sim V$ and $U\cap V\subseteq I\subseteq U\cup V$, then
$U\sim I\sim V$''. The reader can prove this alternative
characterization as an easy exercise in Boolean algebra.
\end{itemize}

Boolean interval partitions have come up in combinatorics before (e.g.,
\cite{BrOlNo09}, \cite{Dawson80}, \cite{DedTit20}, \cite{GorMah97}).

We shall now construct a shade map from any Boolean interval partition:

\begin{theorem}
\label{thm.bip.Bip-to-Shade}Let $E$ be a set. Let $\mathbf{P}$ be a Boolean
interval partition of $\mathcal{P}\left(  E\right)  $.

For any $F\in\mathcal{P}\left(  E\right)  $, let $\left[  \alpha\left(
F\right)  ,\ \tau\left(  F\right)  \right]  $ denote the (unique) block of
$\mathbf{P}$ that contains $F$.

We define a map $\operatorname*{Shade}:\mathcal{P}\left(  E\right)
\rightarrow\mathcal{P}\left(  E\right)  $ by setting%
\[
\operatorname*{Shade}F=E\setminus\left(  \tau\left(  F\right)  \setminus
\alpha\left(  F\right)  \right)  \ \ \ \ \ \ \ \ \ \ \text{for any }%
F\in\mathcal{P}\left(  E\right)  .
\]

Then:

\textbf{(a)} The map $\operatorname*{Shade}$ is a shade map on $E$.

\textbf{(b)} We have $\alpha\left(  F\right)  =F\cap\operatorname*{Shade}F$
and $\tau\left(  F\right)  =F\cup\left(  E\setminus\operatorname*{Shade}%
F\right)  $ for any $F\in\mathcal{P}\left(  E\right)  $.

\textbf{(c)} We have $\mathbf{P}=\left\{  \left[  \alpha\left(  F\right)
,\ \tau\left(  F\right)  \right]  \ \mid\ F\in\mathcal{P}\left(  E\right)
\right\}  $.
\end{theorem}

The proof of Theorem \ref{thm.bip.Bip-to-Shade} is rather easy. We lighten our
burden somewhat with a simple lemma (which can be easily checked using Venn diagrams):

\begin{lemma}
\label{lem.bip.Bip-to-Shade.lemb}Let $X$, $Y$, $Z$ and $E$ be four sets such
that $X\subseteq Y\subseteq Z\subseteq E$. Then:

\textbf{(a)} We have $Y\cup\left(  Z\setminus X\right)  =Z$.

\textbf{(b)} We have $Y\cap\left(  E\setminus\left(  Z\setminus X\right)
\right)  =X$.
\end{lemma}

\begin{proof}
[of Theorem \ref{thm.bip.Bip-to-Shade}]\textbf{(a)} We shall prove the
following two statements:

\begin{statement}
\textit{Statement 1:} If $F\in\mathcal{P}\left(  E\right)  $ and $u\in
E\setminus\operatorname*{Shade}F$, then $\operatorname*{Shade}\left(
F\cup\left\{  u\right\}  \right)  =\operatorname*{Shade}F$.
\end{statement}

\begin{statement}
\textit{Statement 2:} If $F\in\mathcal{P}\left(  E\right)  $ and $u\in
E\setminus\operatorname*{Shade}F$, then $\operatorname*{Shade}\left(
F\setminus\left\{  u\right\}  \right)  =\operatorname*{Shade}F$.
\end{statement}

[\textit{Proof of Statement 1:} Let $F\in\mathcal{P}\left(  E\right)  $ and
$u\in E\setminus\operatorname*{Shade}F$. We must prove that
$\operatorname*{Shade}\left(  F\cup\left\{  u\right\}  \right)
=\operatorname*{Shade}F$.

The definition of $\alpha\left(  F\right)  $ and $\tau\left(  F\right)  $
reveals that $\left[  \alpha\left(  F\right)  ,\ \tau\left(  F\right)
\right]  $ is the (unique) block of $\mathbf{P}$ that contains $F$. Thus,
$\left[  \alpha\left(  F\right)  ,\ \tau\left(  F\right)  \right]  $ is a
block of $\mathbf{P}$ and contains $F$.

The definition of $\operatorname*{Shade}$ yields $\operatorname*{Shade}%
F=E\setminus\left(  \tau\left(  F\right)  \setminus\alpha\left(  F\right)
\right)  $. Thus, $E\setminus\operatorname*{Shade}F=\tau\left(  F\right)
\setminus\alpha\left(  F\right)  $ (since $\tau\left(  F\right)
\setminus\alpha\left(  F\right)  $ is a subset of $E$). Now, $u\in
E\setminus\operatorname*{Shade}F=\tau\left(  F\right)  \setminus\alpha\left(
F\right)  $. In other words, $u\in\tau\left(  F\right)  $ and $u\notin%
\alpha\left(  F\right)  $.

On the other hand, the Boolean interval $\left[  \alpha\left(  F\right)
,\ \tau\left(  F\right)  \right]  $ contains $F$. In other words,
$\alpha\left(  F\right)  \subseteq F\subseteq\tau\left(  F\right)  $.

Now, set $F^{\prime}=F\cup\left\{  u\right\}  $. From $F\subseteq\tau\left(
F\right)  $ and $u\in\tau\left(  F\right)  $, we thus obtain $F^{\prime
}\subseteq\tau\left(  F\right)  $. Combined with $\alpha\left(  F\right)
\subseteq F\subseteq F^{\prime}$, this entails $F^{\prime}\in\left[
\alpha\left(  F\right)  ,\ \tau\left(  F\right)  \right]  $. Hence, $\left[
\alpha\left(  F\right)  ,\ \tau\left(  F\right)  \right]  $ is a block of
$\mathbf{P}$ that contains $F^{\prime}$ (since we already know that $\left[
\alpha\left(  F\right)  ,\ \tau\left(  F\right)  \right]  $ is a block of
$\mathbf{P}$).

However, the definition of $\alpha\left(  F^{\prime}\right)  $ and
$\tau\left(  F^{\prime}\right)  $ shows that $\left[  \alpha\left(  F^{\prime
}\right)  ,\ \tau\left(  F^{\prime}\right)  \right]  $ is the (unique) block
of $\mathbf{P}$ that contains $F^{\prime}$. Since we know that $\left[
\alpha\left(  F\right)  ,\ \tau\left(  F\right)  \right]  $ is a block of
$\mathbf{P}$ that contains $F^{\prime}$, we therefore conclude that $\left[
\alpha\left(  F^{\prime}\right)  ,\ \tau\left(  F^{\prime}\right)  \right]
=\left[  \alpha\left(  F\right)  ,\ \tau\left(  F\right)  \right]  $. Hence,
we have%
\[
\alpha\left(  F^{\prime}\right)  =\alpha\left(  F\right)
\ \ \ \ \ \ \ \ \ \ \text{and}\ \ \ \ \ \ \ \ \ \ \tau\left(  F^{\prime
}\right)  =\tau\left(  F\right)
\]
(because a Boolean interval $\left[  U,\ V\right]  $ uniquely determines both
$U$ and $V$). Now, the definition of $\operatorname*{Shade}$ yields%
\[
\operatorname*{Shade}\left(  F^{\prime}\right)  =E\setminus\left(
\underbrace{\tau\left(  F^{\prime}\right)  }_{=\tau\left(  F\right)
}\setminus\underbrace{\alpha\left(  F^{\prime}\right)  }_{=\alpha\left(
F\right)  }\right)  =E\setminus\left(  \tau\left(  F\right)  \setminus
\alpha\left(  F\right)  \right)  =\operatorname*{Shade}F.
\]
In view of $F^{\prime}=F\cup\left\{  u\right\}  $, this rewrites as
$\operatorname*{Shade}\left(  F\cup\left\{  u\right\}  \right)
=\operatorname*{Shade}F$. This proves Statement 1.]

[\textit{Proof of Statement 2:} Let $F\in\mathcal{P}\left(  E\right)  $ and
$u\in E\setminus\operatorname*{Shade}F$. We must prove that
$\operatorname*{Shade}\left(  F\setminus\left\{  u\right\}  \right)
=\operatorname*{Shade}F$.

We proceed exactly as in our above proof of Statement 1 up until the point
where we define $F^{\prime}$. Insead of setting $F^{\prime}=F\cup\left\{
u\right\}  $, we now set $F^{\prime}=F\setminus\left\{  u\right\}  $.
Combining $\alpha\left(  F\right)  \subseteq F$ with $u\notin\alpha\left(
F\right)  $, we obtain $\alpha\left(  F\right)  \subseteq F\setminus\left\{
u\right\}  =F^{\prime}$. Combining this with $F^{\prime}=F\setminus\left\{
u\right\}  \subseteq F\subseteq\tau\left(  F\right)  $, we see that
$F^{\prime}\in\left[  \alpha\left(  F\right)  ,\ \tau\left(  F\right)
\right]  $. From this, we can obtain $\operatorname*{Shade}\left(  F^{\prime
}\right)  =\operatorname*{Shade}F$ by the same argument that we used back in
the proof of Statement 1. In view of $F^{\prime}=F\setminus\left\{  u\right\}
$, this rewrites as $\operatorname*{Shade}\left(  F\setminus\left\{
u\right\}  \right)  =\operatorname*{Shade}F$. This proves Statement 2.]

Now, we have proved Statements 1 and 2. Thus, the map $\operatorname*{Shade}%
:\mathcal{P}\left(  E\right)  \rightarrow\mathcal{P}\left(  E\right)  $
satisfies the two axioms in Definition \ref{def.Shade.shademap}. In other
words, this map is a shade map. This proves Theorem \ref{thm.bip.Bip-to-Shade}
\textbf{(a)}.

\textbf{(b)} Let $F\in\mathcal{P}\left(  E\right)  $. We must prove that
$\alpha\left(  F\right)  =F\cap\operatorname*{Shade}F$ and $\tau\left(
F\right)  =F\cup\left(  E\setminus\operatorname*{Shade}F\right)  $.

As in the above proof of Statement 1, we can see that $\alpha\left(  F\right)
\subseteq F\subseteq\tau\left(  F\right)  $ and $\operatorname*{Shade}%
F=E\setminus\left(  \tau\left(  F\right)  \setminus\alpha\left(  F\right)
\right)  $ and $E\setminus\operatorname*{Shade}F=\tau\left(  F\right)
\setminus\alpha\left(  F\right)  $. Hence, Lemma
\ref{lem.bip.Bip-to-Shade.lemb} \textbf{(b)} (applied to $X=\alpha\left(
F\right)  $ and $Y=F$ and $Z=\tau\left(  F\right)  $) yields that
$F\cap\left(  E\setminus\left(  \tau\left(  F\right)  \setminus\alpha\left(
F\right)  \right)  \right)  =\alpha\left(  F\right)  $. In view of
$\operatorname*{Shade}F=E\setminus\left(  \tau\left(  F\right)  \setminus
\alpha\left(  F\right)  \right)  $, this rewrites as $F\cap
\operatorname*{Shade}F=\alpha\left(  F\right)  $. In other words,
$\alpha\left(  F\right)  =F\cap\operatorname*{Shade}F$.

Furthermore, Lemma \ref{lem.bip.Bip-to-Shade.lemb} \textbf{(a)} (applied to
$X=\alpha\left(  F\right)  $ and $Y=F$ and $Z=\tau\left(  F\right)  $) yields
that $F\cup\left(  \tau\left(  F\right)  \setminus\alpha\left(  F\right)
\right)  =\tau\left(  F\right)  $. In view of $E\setminus\operatorname*{Shade}%
F=\tau\left(  F\right)  \setminus\alpha\left(  F\right)  $, this rewrites as
$F\cup\left(  E\setminus\operatorname*{Shade}F\right)  =\tau\left(  F\right)
$. In other words, $\tau\left(  F\right)  =F\cup\left(  E\setminus
\operatorname*{Shade}F\right)  $. Thus, Theorem \ref{thm.bip.Bip-to-Shade}
\textbf{(b)} is proven.

\textbf{(c)} Each block of $\mathbf{P}$ has the form $\left[  \alpha\left(
F\right)  ,\ \tau\left(  F\right)  \right]  $ for some $F\in\mathcal{P}\left(
E\right)  $ (since it is a Boolean interval, thus nonempty, therefore contains
some $F\in\mathcal{P}\left(  E\right)  $; but then it must be $\left[
\alpha\left(  F\right)  ,\ \tau\left(  F\right)  \right]  $ for this $F$).
Conversely, any set of the form $\left[  \alpha\left(  F\right)
,\ \tau\left(  F\right)  \right]  $ is a block of $\mathbf{P}$ (by the
definition of $\left[  \alpha\left(  F\right)  ,\ \tau\left(  F\right)
\right]  $). Combining these two facts, we conclude that the blocks of
$\mathbf{P}$ are precisely the sets of the form $\left[  \alpha\left(
F\right)  ,\ \tau\left(  F\right)  \right]  $ with $F\in\mathcal{P}\left(
E\right)  $. But this is precisely the claim of Theorem
\ref{thm.bip.Bip-to-Shade} \textbf{(c)}.
\end{proof}

A converse to Theorem \ref{thm.bip.Bip-to-Shade} is provided by the following theorem:

\begin{theorem}
\label{thm.bip.Shade-to-Bip}Let $E$ be a finite set. Let
$\operatorname*{Shade}:\mathcal{P}\left(  E\right)  \rightarrow\mathcal{P}%
\left(  E\right)  $ be a shade map on $E$. Define a map $\alpha:\mathcal{P}%
\left(  E\right)  \rightarrow\mathcal{P}\left(  E\right)  $ by setting%
\[
\alpha\left(  F\right)  =F\cap\operatorname*{Shade}%
F\ \ \ \ \ \ \ \ \ \ \text{for any }F\in\mathcal{P}\left(  E\right)  .
\]
Define a map $\tau:\mathcal{P}\left(  E\right)  \rightarrow\mathcal{P}\left(
E\right)  $ by setting%
\[
\tau\left(  F\right)  =F\cup\left(  E\setminus\operatorname*{Shade}F\right)
\ \ \ \ \ \ \ \ \ \ \text{for any }F\in\mathcal{P}\left(  E\right)  .
\]
Let
\[
\mathbf{P}=\left\{  \left[  \alpha\left(  F\right)  ,\ \tau\left(  F\right)
\right]  \ \mid\ F\in\mathcal{P}\left(  E\right)  \right\}  .
\]

Then:

\textbf{(a)} We have $F\in\left[  \alpha\left(  F\right)  ,\ \tau\left(
F\right)  \right]  $ for any $F\in\mathcal{P}\left(  E\right)  $.

\textbf{(b)} If $F\in\mathcal{P}\left(  E\right)  $ and $G\in\left[
\alpha\left(  F\right)  ,\ \tau\left(  F\right)  \right]  $, then $\left[
\alpha\left(  F\right)  ,\ \tau\left(  F\right)  \right]  =\left[
\alpha\left(  G\right)  ,\ \tau\left(  G\right)  \right]  $.

\textbf{(c)} The set $\mathbf{P}$ is a Boolean interval partition of
$\mathcal{P}\left(  E\right)  $.

\textbf{(d)} We have $\operatorname*{Shade}F=E\setminus\left(  \tau\left(
F\right)  \setminus\alpha\left(  F\right)  \right)  $ for any $F\in
\mathcal{P}\left(  E\right)  $.
\end{theorem}

To prove this theorem, we will need the following variant of Lemma
\ref{lem.Shade.FuG}:

\begin{lemma}
\label{lem.Shade.FoG}Let $E$ be a set. Let $\operatorname*{Shade}%
:\mathcal{P}\left(  E\right)  \rightarrow\mathcal{P}\left(  E\right)  $ be a
shade map. Let $A$ and $B$ be two subsets of $E$ such that $B$ is finite and
$B\cap\operatorname*{Shade}A=\varnothing$. Then, $\operatorname*{Shade}\left(
A\setminus B\right)  =\operatorname*{Shade}A$.
\end{lemma}

\begin{proof}
[of Lemma \ref{lem.Shade.FoG}]Induction on $\left\vert B\right\vert $, using
Axiom 2 from Definition \ref{def.Shade}.
\end{proof}

We will also use another simple set-theoretical lemma (easily checked using
Venn diagrams):

\begin{lemma}
\label{lem.bip.Shade-to-Bip.lem1}Let $E$ be a set. Let $X$ and $Y$ be two
subsets of $E$. Then,%
\[
E\setminus Y=\left(  X\cup\left(  E\setminus Y\right)  \right)  \setminus
\left(  X\cap Y\right)  .
\]

\end{lemma}

\begin{proof}
[of Theorem \ref{thm.bip.Shade-to-Bip}]\textbf{(a)} Let $F\in\mathcal{P}%
\left(  E\right)  $. The definition of $\alpha$ yields $\alpha\left(
F\right)  =F\cap\operatorname*{Shade}F\subseteq F$. The definition of $\tau$
yields $\tau\left(  F\right)  =F\cup\left(  E\setminus\operatorname*{Shade}%
F\right)  \supseteq F$, so that $F\subseteq\tau\left(  F\right)  $. Thus,
$\alpha\left(  F\right)  \subseteq F\subseteq\tau\left(  F\right)  $. In other
words, $F\in\left[  \alpha\left(  F\right)  ,\ \tau\left(  F\right)  \right]
$. This proves Theorem \ref{thm.bip.Shade-to-Bip} \textbf{(a)}.

\textbf{(b)} Let $F\in\mathcal{P}\left(  E\right)  $ and $G\in\left[
\alpha\left(  F\right)  ,\ \tau\left(  F\right)  \right]  $. From $G\in\left[
\alpha\left(  F\right)  ,\ \tau\left(  F\right)  \right]  $, we obtain
$\alpha\left(  F\right)  \subseteq G\subseteq\tau\left(  F\right)  $. Thus,
$G\subseteq\tau\left(  F\right)  =F\cup\left(  E\setminus\operatorname*{Shade}%
F\right)  $ (by the definition of $\tau$). From this, straightforward
set-theoretical reasoning leads to%
\begin{equation}
\left(  G\setminus F\right)  \cap\operatorname*{Shade}F=\varnothing.
\label{pf.thm.bip.Shade-to-Bip.b.1}%
\end{equation}

Furthermore, the definition of $\alpha$ yields $F\cap\operatorname*{Shade}%
F=\alpha\left(  F\right)  \subseteq G$. From this, straightforward
set-theoretical reasoning leads to%
\begin{equation}
\left(  F\setminus G\right)  \cap\operatorname*{Shade}F=\varnothing.
\label{pf.thm.bip.Shade-to-Bip.b.2}%
\end{equation}
Hence, Lemma \ref{lem.Shade.FoG} (applied to $A=F$ and $B=F\setminus G$)
yields $\operatorname*{Shade}\left(  F\setminus\left(  F\setminus G\right)
\right)  =\operatorname*{Shade}F$. In view of $F\setminus\left(  F\setminus
G\right)  =F\cap G$, this rewrites as $\operatorname*{Shade}\left(  F\cap
G\right)  =\operatorname*{Shade}F$. Hence, (\ref{pf.thm.bip.Shade-to-Bip.b.1})
rewrites as
\[
\left(  G\setminus F\right)  \cap\operatorname*{Shade}\left(  F\cap G\right)
=\varnothing.
\]
Thus, Lemma \ref{lem.Shade.FuG} (applied to $A=F\cap G$ and $B=G\setminus F$)
yields \newline$\operatorname*{Shade}\left(  \left(  F\cap G\right)
\cup\left(  G\setminus F\right)  \right)  =\operatorname*{Shade}\left(  F\cap
G\right)  =\operatorname*{Shade}F$. In view of $\left(  F\cap G\right)
\cup\left(  G\setminus F\right)  =G$, this rewrites as
\[
\operatorname*{Shade}G=\operatorname*{Shade}F.
\]

By further set-theoretical reasoning, using the equalities
(\ref{pf.thm.bip.Shade-to-Bip.b.1}) and (\ref{pf.thm.bip.Shade-to-Bip.b.2})
(and nothing else specific to the sets $F$, $G$ and $\operatorname*{Shade}F$),
we can easily obtain
\[
F\cap\operatorname*{Shade}F=G\cap\operatorname*{Shade}%
F\ \ \ \ \ \ \ \ \ \ \text{and}\ \ \ \ \ \ \ \ \ \ F\cup\left(  E\setminus
\operatorname*{Shade}F\right)  =G\cup\left(  E\setminus\operatorname*{Shade}%
F\right)  .
\]
In view of%
\begin{align*}
\alpha\left(  F\right)   &  =F\cap\operatorname*{Shade}%
F\ \ \ \ \ \ \ \ \ \ \text{and}\ \ \ \ \ \ \ \ \ \ \alpha\left(  G\right)
=G\cap\underbrace{\operatorname*{Shade}G}_{=\operatorname*{Shade}F}%
=G\cap\operatorname*{Shade}F\ \ \ \ \ \ \ \ \ \ \text{and}\\
\tau\left(  F\right)   &  =F\cup\left(  E\setminus\operatorname*{Shade}%
F\right)  \ \ \ \ \ \ \ \ \ \ \text{and}\ \ \ \ \ \ \ \ \ \ \tau\left(
G\right)  =G\cup\left(  E\setminus\underbrace{\operatorname*{Shade}%
G}_{=\operatorname*{Shade}F}\right)  =G\cup\left(  E\setminus
\operatorname*{Shade}F\right)  ,
\end{align*}
we can rewrite these two equalities as $\alpha\left(  F\right)  =\alpha\left(
G\right)  $ and $\tau\left(  F\right)  =\tau\left(  G\right)  $. Thus,
$\left[  \alpha\left(  F\right)  ,\ \tau\left(  F\right)  \right]  =\left[
\alpha\left(  G\right)  ,\ \tau\left(  G\right)  \right]  $. Theorem
\ref{thm.bip.Shade-to-Bip} \textbf{(b)} is thus proven.

\textbf{(c)} Clearly, $\mathbf{P}$ is a set of Boolean intervals of
$\mathcal{P}\left(  E\right)  $. Moreover, the Boolean intervals $\left[
\alpha\left(  F\right)  ,\ \tau\left(  F\right)  \right]  \in\mathbf{P}$ are
pairwise disjoint (because if two of these intervals have some element $G$ in
common, then Theorem \ref{thm.bip.Shade-to-Bip} \textbf{(b)} yields that they
must both equal $\left[  \alpha\left(  G\right)  ,\ \beta\left(  G\right)
\right]  $), and their union is $\mathcal{P}\left(  E\right)  $ (since every
$F\in\mathcal{P}\left(  E\right)  $ satisfies $F\in\left[  \alpha\left(
F\right)  ,\ \tau\left(  F\right)  \right]  $ by Theorem
\ref{thm.bip.Shade-to-Bip} \textbf{(a)}). Thus, $\mathbf{P}$ is a Boolean
interval partition of $\mathcal{P}\left(  E\right)  $. This proves Theorem
\ref{thm.bip.Shade-to-Bip} \textbf{(c)}.

\textbf{(d)} Let $F\in\mathcal{P}\left(  E\right)  $. Then, Lemma
\ref{lem.bip.Shade-to-Bip.lem1} (applied to $X=F$ and $Y=\operatorname*{Shade}%
F$) yields%
\[
E\setminus\operatorname*{Shade}F=\underbrace{\left(  F\cup\left(
E\setminus\operatorname*{Shade}F\right)  \right)  }_{\substack{=\tau\left(
F\right)  \\\text{(since }\tau\left(  F\right)  =F\cup\left(  E\setminus
\operatorname*{Shade}F\right)  \text{)}}}\setminus\underbrace{\left(
F\cap\operatorname*{Shade}F\right)  }_{\substack{=\alpha\left(  F\right)
\\\text{(since }\alpha\left(  F\right)  =F\cap\operatorname*{Shade}F\text{)}%
}}=\tau\left(  F\right)  \setminus\alpha\left(  F\right)  .
\]
However, $\operatorname*{Shade}F$ is a subset of $E$; thus,%
\[
\operatorname*{Shade}F=E\setminus\underbrace{\left(  E\setminus
\operatorname*{Shade}F\right)  }_{=\tau\left(  F\right)  \setminus
\alpha\left(  F\right)  }=E\setminus\left(  \tau\left(  F\right)
\setminus\alpha\left(  F\right)  \right)  .
\]
This proves Theorem \ref{thm.bip.Shade-to-Bip} \textbf{(d)}.
\end{proof}

Combining Theorem \ref{thm.bip.Bip-to-Shade} with Theorem
\ref{thm.bip.Shade-to-Bip}, we obtain the following:

\begin{theorem}
\label{thm.bip.cryptomor}Let $E$ be a finite set. Then, there is a bijection
from the set%
\[
\left\{  \text{shade maps }\operatorname*{Shade}:\mathcal{P}\left(  E\right)
\rightarrow\mathcal{P}\left(  E\right)  \right\}
\]
to the set%
\[
\left\{  \text{Boolean interval partitions of }\mathcal{P}\left(  E\right)
\right\}  .
\]
It sends each map $\operatorname*{Shade}$ to the Boolean interval partition
$\mathbf{P}$ defined in Theorem \ref{thm.bip.Shade-to-Bip}\textbf{ (c)}. Its
inverse map sends each Boolean interval partition $\mathbf{P}$ to the map
$\operatorname*{Shade}$ defined in Theorem \ref{thm.bip.Bip-to-Shade}.
\end{theorem}

\begin{proof}
[of Theorem \ref{thm.bip.cryptomor}]This follows easily from Theorem
\ref{thm.bip.Bip-to-Shade} and Theorem \ref{thm.bip.Shade-to-Bip}.
\end{proof}

\begin{question}
According to Theorem \ref{thm.antimat.Shade}, any antimatroid gives rise to an
inclusion-reversing shade map, which in turn gives rise to a Boolean interval
partition by Theorem \ref{thm.bip.Shade-to-Bip}. Is this construction
equivalent to the construction of a Boolean interval partition from an
antimatroid described by Gordon and McMahon in \cite[Theorem 2.5]{GorMah97}?
\end{question}

\section{The topological viewpoint}

Now we return to the setting of Section \ref{sec.elser}. We aim to reinterpret
Theorem \ref{thm.elser-Shade0} in the terms of combinatorial topology
(specifically, finite simplicial complexes) and strengthen it. We recall the
definition of a \emph{simplicial complex}:\footnote{We forget all the
conventions we have introduced so far. (Thus, for example, $E$ no longer means
the edge set of a graph $\Gamma$.)}

\begin{definition}
Let $E$ be a finite set. A \emph{simplicial complex} on ground set $E$ means a
subset $\mathcal{A}$ of the power set of $E$ with the following property:

\begin{statement}
If $P\in\mathcal{A}$ and $Q\subseteq P$, then $Q\in\mathcal{A}$.
\end{statement}
\end{definition}

Thus, in terms of posets, a simplicial complex on ground set $E$ means a
down-closed subset of the Boolean lattice on $E$. Note that a simplicial
complex contains the empty set $\varnothing$ unless it is empty itself.

We refer to \cite{Kozlov20} for context and theory about simplicial complexes.
We shall restrict ourselves to the few definitions relevant to what we will
prove. The following is straightforward to check:

\begin{proposition}
\label{prop.elser-sc}Let us use the notations from Section \ref{sec.elser} as
well as Definition \ref{def.Shade}. Let $G$ be any subset of $E$. Let%
\begin{equation}
\mathcal{A}=\left\{  F\subseteq E\ \mid\ G\not \subseteq \operatorname*{Shade}%
F\right\}  . \label{eq.prop.elser-sc.A=}%
\end{equation}
Then, $\mathcal{A}$ is a simplicial complex on ground set $E$.
\end{proposition}

\begin{example}
The following pictures illustrate this simplicial complex on an example.
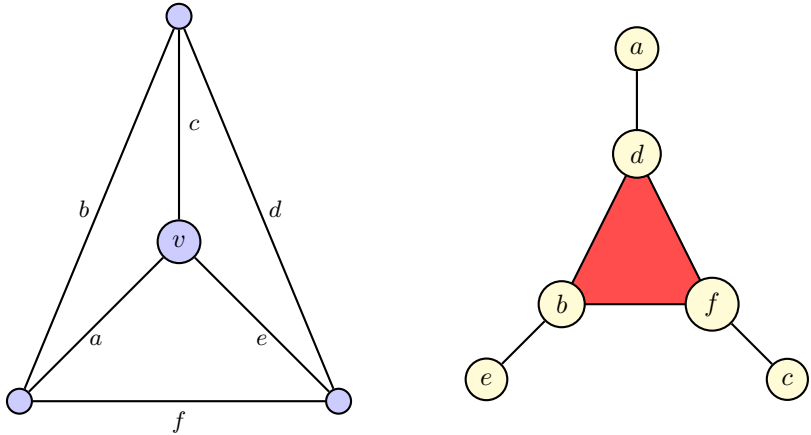
\[%
\begin{tikzpicture}%
[-,auto,node distance=3cm, thick,main node/.style={circle,fill=blue!20,draw}]
\node[main node] (1) {$v$};
\node[main node] [below left of=1] (2) {};
\node[main node] [below right of=1] (4) {};
\node[main node] [above of=1] (3) {};
\path[every node/.style={font=\sffamily\small}]
(1) -- (3) node[midway, right] {$c$}
(3) -- (2) node[midway, left] {$b$}
(2) -- (4) node[midway, below] {$f$}
(4) -- (3) node[midway, right] {$d$}
(2) -- (1) node[midway, below] {$a$}
(1) -- (4) node[midway, below] {$e$};
\path[every node/.style={font=\sffamily\small}] (1) edge
(3) (3) edge (2) (2) edge
(4) (4) edge (3) (2) edge
(1) (1) edge (4);
\end{tikzpicture}
\qquad\qquad\raisebox{15pt}{
\begin{tikzpicture}[thick, main node/.style={circle,fill=yellow!20,draw}]
\filldraw[draw=black, fill=red!70] (-1, -1) -- (0, 1) -- (1, -1) -- cycle;
\node[main node] (a) at (0, 2.4) {$a$};
\node[main node] (b) at (-1, -1) {$b$};
\node[main node] (c) at (2, -2) {$c$};
\node[main node] (d) at (0, 1) {$d$};
\node[main node] (e) at (-2, -2) {$e$};
\node[main node] (f) at (1, -1) {$f$};
\draw(a) edge (d);
\draw(b) edge (e);
\draw(c) edge (f);
\end{tikzpicture}
}
\]
The left picture is a graph $\Gamma$ (with the vertex labelled $v$ playing the
role of $v$), whereas the right picture shows the corresponding simplicial
complex $\mathcal{A}$ for $G=E$ (that is, the simplicial complex whose faces
are the subsets of $E$ that are not pandemic).
\end{example}

To state the main result of this section, we need the following
definition\footnote{We continue using the definitions from Subsection
\ref{subsect.sets}, even though $E$ is now just an arbitrary finite set.}:

\begin{definition}
\label{def.matchings}Let $E$ be a finite set. Let $\mathcal{A}$ be a
simplicial complex on ground set $E$.

\textbf{(a)} A complete matching $\mu$ of $\mathcal{A}$ is said to be
\emph{acyclic} if there exists no tuple $\left(  B_{1},B_{2},\ldots
,B_{n}\right)  $ of distinct sets $B_{1},B_{2},\ldots,B_{n}\in\mathcal{A}$
with the property that $n\geq2$ and that%
\[
\mu\left(  B_{i}\right)  \prec B_{i}\ \ \ \ \ \ \ \ \ \ \text{for each }%
i\in\left\{  1,2,\ldots,n\right\}
\]
and%
\[
\mu\left(  B_{i}\right)  \prec B_{i+1}\ \ \ \ \ \ \ \ \ \ \text{for each }%
i\in\left\{  1,2,\ldots,n-1\right\}
\]
and%
\[
\mu\left(  B_{n}\right)  \prec B_{1}.
\]

\textbf{(b)} The simplicial complex $\mathcal{A}$ is said to be
\emph{collapsible} if it has an acyclic complete matching.
\end{definition}

Note that our definition of an ``acyclic complete
matching'' is a particular case of \cite[Definition
10.7]{Kozlov20}. Our notion of ``collapsible''
is equivalent to the classical notion of ``collapsible''
(even though the latter is usually defined
differently) because of \cite[Theorem 10.9]{Kozlov20}.

We now claim:

\begin{theorem}
\label{thm.elser-coll}Let us use the notations from Section \ref{sec.elser} as
well as Definition \ref{def.Shade}. Let $G$ be any subset of $E$. Define
$\mathcal{A}$ as in (\ref{eq.prop.elser-sc.A=}). Then, the simplicial complex
$\mathcal{A}$ is collapsible.
\end{theorem}

Collapsible simplicial complexes are well-behaved in various ways -- in
particular, they are contractible (\cite[Corollary 9.19]{Kozlov20}), and thus
have trivial homotopy and homology groups (in positive degrees). Moreover, the
reduced Euler characteristic of any collapsible simplicial complex is $0$ (for
obvious reasons: having a complete matching suffices, even if it is not
acyclic); thus, Theorem \ref{thm.elser-Shade0} follows from Theorem
\ref{thm.elser-coll}.

Our proof of Theorem \ref{thm.elser-coll} will rely on the following simple
lemma (whose proof is left as an exercise):

\begin{lemma}
\label{lem.elser-coll.lem}Let $X$ and $Y$ be two sets, and let $u\in X\cap Y$.
If $X\setminus\left\{  u\right\}  \prec Y$, then $X=Y$.
\end{lemma}

\begin{proof}
[of Theorem \ref{thm.elser-coll}]We know from Proposition \ref{prop.elser-sc}
that $\mathcal{A}$ is a simplicial complex. It remains to show that
$\mathcal{A}$ is collapsible.

Note that our $\mathcal{A}$ is precisely the set $\mathcal{A}$ defined in the
proof of Theorem \ref{thm.elser-Shade0} above.

We equip the finite set $E$ with a total order (chosen arbitrarily).

If $F\in\mathcal{A}$, then we define the edge $\varepsilon\left(  F\right)
\in G\setminus\operatorname*{Shade}F$ as in the proof of Theorem
\ref{thm.elser-Shade0}.

We define a complete matching $\mu:\mathcal{A}\rightarrow\mathcal{A}$ of
$\mathcal{A}$ as in the proof of Theorem \ref{thm.elser-Shade0}. We shall now
prove that this complete matching $\mu$ is acyclic.

Indeed, let $\left(  B_{1},B_{2},\ldots,B_{n}\right)  $ be a tuple of distinct
sets $B_{1},B_{2},\ldots,B_{n}\in\mathcal{A}$ with the property that $n\geq2$
and that%
\begin{equation}
\mu\left(  B_{i}\right)  \prec B_{i}\ \ \ \ \ \ \ \ \ \ \text{for each }%
i\in\left\{  1,2,\ldots,n\right\}  \label{pf.thm.elser-coll.cyc0}%
\end{equation}
and%
\begin{equation}
\mu\left(  B_{i}\right)  \prec B_{i+1}\ \ \ \ \ \ \ \ \ \ \text{for each }%
i\in\left\{  1,2,\ldots,n-1\right\}  \label{pf.thm.elser-coll.cyc1}%
\end{equation}
and%
\begin{equation}
\mu\left(  B_{n}\right)  \prec B_{1}. \label{pf.thm.elser-coll.cyc2}%
\end{equation}
We shall derive a contradiction.

Set $B_{n+1}=B_{1}$. Then, combining (\ref{pf.thm.elser-coll.cyc1}) with
(\ref{pf.thm.elser-coll.cyc2}), we conclude that%
\begin{equation}
\mu\left(  B_{i}\right)  \prec B_{i+1}\ \ \ \ \ \ \ \ \ \ \text{for each }%
i\in\left\{  1,2,\ldots,n\right\}  . \label{pf.thm.elser-coll.cyc12}%
\end{equation}

We now claim the following:

\begin{statement}
\textit{Claim 1:} We have $\operatorname*{Shade}\left(  B_{i}\right)
\subseteq\operatorname*{Shade}\left(  B_{i+1}\right)  $ for each $i\in\left\{
1,2,\ldots,n\right\}  $.
\end{statement}

[\textit{Proof of Claim 1:} Let $i\in\left\{  1,2,\ldots,n\right\}  $. From
(\ref{pf.thm.elser-coll.cyc0}), we see that $\mu\left(  B_{i}\right)  \prec
B_{i}$, so that $\varepsilon\left(  B_{i}\right)  \in B_{i}$ (by the
definition of $\mu$). Set $u=\varepsilon\left(  B_{i}\right)  $. Then, the
definition of $\mu$ yields $\mu\left(  B_{i}\right)  =B_{i}\setminus\left\{
u\right\}  $.

However, $u=\varepsilon\left(  B_{i}\right)  \in G\setminus
\operatorname*{Shade}\left(  B_{i}\right)  $ (by the definition of
$\varepsilon\left(  B_{i}\right)  $). In other words, $u\in G$ and
$u\notin\operatorname*{Shade}\left(  B_{i}\right)  $. Therefore,
(\ref{eq.lem.Shade-tog.diff}) (applied to $F=B_{i}$) yields
$\operatorname*{Shade}\left(  B_{i}\setminus\left\{  u\right\}  \right)
=\operatorname*{Shade}\left(  B_{i}\right)  $. This can be rewritten as
$\operatorname*{Shade}\left(  \mu\left(  B_{i}\right)  \right)
=\operatorname*{Shade}\left(  B_{i}\right)  $ (since $\mu\left(  B_{i}\right)
=B_{i}\setminus\left\{  u\right\}  $).

But (\ref{pf.thm.elser-coll.cyc12}) yields $\mu\left(  B_{i}\right)  \prec
B_{i+1}$, so that $\mu\left(  B_{i}\right)  \subseteq B_{i+1}$ and thus
$\operatorname*{Shade}\left(  \mu\left(  B_{i}\right)  \right)  \subseteq
\operatorname*{Shade}\left(  B_{i+1}\right)  $ (by Lemma
\ref{lem.Shade-monoton}). In view of $\operatorname*{Shade}\left(  \mu\left(
B_{i}\right)  \right)  =\operatorname*{Shade}\left(  B_{i}\right)  $, this can
be rewritten as $\operatorname*{Shade}\left(  B_{i}\right)  \subseteq
\operatorname*{Shade}\left(  B_{i+1}\right)  $. This proves Claim 1.]

Claim 1 shows that $\operatorname*{Shade}\left(  B_{1}\right)  \subseteq
\operatorname*{Shade}\left(  B_{2}\right)  \subseteq\cdots\subseteq
\operatorname*{Shade}\left(  B_{n}\right)  \subseteq\operatorname*{Shade}%
\left(  B_{n+1}\right)  $. This chain of inclusions is circular (since
$B_{n+1}=B_{1}$); thus, it must be a chain of equalities. Hence,
$\operatorname*{Shade}\left(  B_{n}\right)  =\operatorname*{Shade}\left(
B_{1}\right)  $. Thus, $\varepsilon\left(  B_{n}\right)  =\varepsilon\left(
B_{1}\right)  $ (since $\varepsilon\left(  F\right)  $ depends only on
$\operatorname*{Shade}F$, not on $F$ itself).

Set $u=\varepsilon\left(  B_{n}\right)  $. Thus, $u=\varepsilon\left(
B_{n}\right)  =\varepsilon\left(  B_{1}\right)  $.

We have $\mu\left(  B_{n}\right)  \prec B_{n}$ (by
(\ref{pf.thm.elser-coll.cyc0})); thus, the definition of $\mu$ yields
$\varepsilon\left(  B_{n}\right)  \in B_{n}$ and $\mu\left(  B_{n}\right)
=B_{n}\setminus\left\{  \varepsilon\left(  B_{n}\right)  \right\}
=B_{n}\setminus\left\{  u\right\}  $ (since $\varepsilon\left(  B_{n}\right)
=u$).

However, $u=\varepsilon\left(  B_{n}\right)  \in B_{n}$ and similarly $u\in
B_{1}$. Thus, $u\in B_{n}\cap B_{1}$. Also, $B_{n}\setminus\left\{  u\right\}
=\mu\left(  B_{n}\right)  \prec B_{1}$ (by (\ref{pf.thm.elser-coll.cyc2})).
Hence, Lemma \ref{lem.elser-coll.lem} (applied to $X=B_{n}$ and $Y=B_{1}$)
yields $B_{n}=B_{1}$. This contradicts the fact that the sets $B_{1}%
,B_{2},\ldots,B_{n}$ are distinct (since $n\geq2$).

Forget that we fixed $\left(  B_{1},B_{2},\ldots,B_{n}\right)  $. We thus have
found a contradiction whenever $\left(  B_{1},B_{2},\ldots,B_{n}\right)  $ is
a tuple of distinct sets $B_{1},B_{2},\ldots,B_{n}\in\mathcal{A}$ with the
property that $n\geq2$ and that (\ref{pf.thm.elser-coll.cyc0}) and
(\ref{pf.thm.elser-coll.cyc1}) and (\ref{pf.thm.elser-coll.cyc2}). Hence,
there exists no such tuple. In other words, the complete matching $\mu$ is
acyclic. Therefore, the simplicial complex $\mathcal{A}$ is collapsible (by
Definition \ref{def.matchings} \textbf{(b)}). This finishes the proof of
Theorem \ref{thm.elser-coll}.
\end{proof}

The analogues of Proposition \ref{prop.elser-sc} and of Theorem
\ref{thm.elser-coll} for vertex-infection (instead of usual infection) also
hold (with the same proofs). More generally, Proposition \ref{prop.elser-sc}
and Theorem \ref{thm.elser-coll} can be generalized to any
inclusion-preserving shade map:

\begin{theorem}
\label{thm.Shade.coll}Let $E$ be any set. Let $\operatorname*{Shade}%
:\mathcal{P}\left(  E\right)  \rightarrow\mathcal{P}\left(  E\right)  $ be an
inclusion-preserving shade map on $E$. Let $G$ be any subset of $E$. Let%
\[
\mathcal{A}=\left\{  F\subseteq E\ \mid\ G\not \subseteq \operatorname*{Shade}%
F\right\}  .
\]
Then:

\textbf{(a)} This $\mathcal{A}$ is a simplicial complex on ground set $E$.

\textbf{(b)} This simplicial complex $\mathcal{A}$ is collapsible.
\end{theorem}

\begin{proof}
Part \textbf{(a)} is a straightforward generalization of Proposition
\ref{prop.elser-sc}, while part \textbf{(b)} is a straightforward
generalization of Theorem \ref{thm.elser-coll}. The proofs we gave above
generalize (mutatis mutandis).
\end{proof}

However, Theorem \ref{thm.Shade.coll} cannot be lifted to the full generality
of arbitrary shade maps, since $\mathcal{A}$ will generally not be a
simplicial complex unless the shade map is inclusion-preserving. (However, for
inclusion-reversing shade maps, we can obtain a variant of Theorem
\ref{thm.Shade.coll} by applying Theorem \ref{thm.Shade.coll} to the dual
shade map $\operatorname*{Shade}\nolimits^{\prime}$ from Proposition
\ref{prop.Shade.dual}.)

\section{Open questions}

I shall now comment on two natural directions of research so far unexplored.

\subsection{The Alexander dual}

Any simplicial complex has an \emph{Alexander dual}, which is defined as follows:

\begin{definition}
Let $E$ be a finite set. Let $\mathcal{A}$ be a simplicial complex on ground
set $E$. Then, we define a new simplicial complex $\mathcal{A}^{\vee}$ on
ground set $E$ by%
\[
\mathcal{A}^{\vee}=\left\{  F\subseteq E\ \mid\ E\setminus F\notin%
\mathcal{A}\right\}  .
\]
(That is, $\mathcal{A}^{\vee}$ consists of those subsets of $E$ whose
complements don't belong to $\mathcal{A}$.) This simplicial complex
$\mathcal{A}^{\vee}$ is called the \emph{Alexander dual} of $\mathcal{A}$.
\end{definition}

It is well-known that a simplicial complex $\mathcal{A}$ and its Alexander
dual $\mathcal{A}^{\vee}$ share many properties; in particular, the reduced
homology of $\mathcal{A}$ is isomorphic to the reduced cohomology of
$\mathcal{A}^{\vee}$ (see, e.g., \cite[Theorem 1.1]{BjoTan09}). However, the
collapsibility and the homotopy types of $\mathcal{A}$ and $\mathcal{A}^{\vee
}$ are not always related. Thus, the following question is suggested but not
answered by Theorem \ref{thm.elser-coll}:

\begin{question}
\label{quest.elser-dual}Let us use the notations from Section \ref{sec.elser}
as well as Definition \ref{def.Shade}. Let $G$ be any subset of $E$. Define
$\mathcal{A}$ as in (\ref{eq.prop.elser-sc.A=}). Is the simplicial complex
\[
\mathcal{A}^{\vee}=\left\{  F\subseteq E\ \mid\ G\subseteq
\operatorname*{Shade}\left(  E\setminus F\right)  \right\}
\]
collapsible? Is it contractible?
\end{question}

\subsection{Several vertices $v$}

Elser's nuclei-based viewpoint in \cite{Elser84} (and \cite[Conjecture
9.1]{DHLetc19}) suggests yet another question.

Our definition of $\operatorname*{Shade}F$ (Definition \ref{def.Shade}), and
the underlying notion of ``infecting'' an
edge, implicitly relied on the choice of vertex $v$. It thus is advisable to
rename the set $\operatorname*{Shade}F$ as $\operatorname*{Shade}%
\nolimits_{v}F$ and combine such sets for different values of $v$. In
particular, we can define%
\[
\mathcal{A}_{U}=\left\{  F\subseteq E\ \mid\ G\not \subseteq
\operatorname*{Shade}\nolimits_{v}F\text{ for some }v\in U\right\}
\]
for any subset $U$ of $V$. This $\mathcal{A}_{U}$ is a simplicial complex
(being the union of a family of simplicial complexes), and thus we can ask the
same questions about it as we did about $\mathcal{A}$:

\begin{question}
\label{quest.elser-multi}What can we say about the homotopy and discrete Morse
theory of $\mathcal{A}_{U}$ ? What about its Alexander dual?
\end{question}

For $G=E$ and $\left\vert U\right\vert >0$, this simplicial complex
$\mathcal{A}_{U}$ is the Alexander dual of the ``$U$-nucleus
complex'' $\Delta_{U}^{G}$ from \cite[Definition
3.2]{DHLetc19} (when $G$ is connected). If \cite[Conjecture 9.1 for
$\left\vert U\right\vert >1$]{DHLetc19} is correct, then the homology of
$\mathcal{A}_{U}$ with real coefficients should be concentrated in a single
degree; this suggests the possible existence of an acyclic partial matching
with all critical faces in one degree.

\acknowledgements\label{sec:ack} I thank Anders Bj\"{o}rner, Galen
Dorpalen-Barry, Dmitry Feichtner-Kozlov, Patricia Hersh, Sam Hopkins, Vic
Reiner, Tom Roby and Richard Stanley for insightful conversations. A referee
has greatly simplified the proof of Theorem \ref{thm.elser-coll}. This
research was supported through the programme ``Oberwolfach
Leibniz Fellows'' by the Mathematisches Forschungsinstitut
Oberwolfach in 2020. I am deeply grateful to the Institute for its hospitality
during a none-too-hospitable time.

\nocite{*}
\bibliographystyle{abbrvnat}
\bibliography{elsersum-journal}
\label{sec:biblio}

\end{document}